\documentclass[onefignum,onetabnum]{siamart171218}

\newcommand{\revise}[1]{{#1}}

\usepackage{amsfonts}
\usepackage{graphicx}
\usepackage{epstopdf}
\usepackage{algorithmic}
\usepackage{tikz}
\usepackage{slashbox}
\ifpdf
  \DeclareGraphicsExtensions{.eps,.pdf,.png,.jpg}
\else
  \DeclareGraphicsExtensions{.eps}
\fi

\newsiamremark{remark}{Remark}
\newsiamremark{hypothesis}{Hypothesis}
\crefname{hypothesis}{Hypothesis}{Hypotheses}
\newsiamthm{claim}{Claim}

\headers{LFA of BDDC Algorithms}{Jed Brown, Yunhui He and Scott MacLachlan}

\title{Local Fourier analysis of Balancing Domain Decomposition by Constraints algorithms\thanks{Received by the editors editors May 31, 2018; accepted for publication (in revised form) June 18, 2019;
published electronically October 29, 2019.\\
\url{https://doi.org/10.1137/18M1191373}
\funding{The work of J.B. was partially funded by U.S. Department of Energy, Office of
Science, Office of Advanced Scientific Computing Research under Award
Number DE-SC0016140. The work of S.M. was partially funded by an NSERC Discovery Grant.}}}

\author{Jed Brown\thanks{Department of Computer Science,
University of Colorado Boulder,
430 UCB, Boulder, CO 80309
  (\email{jed@jedbrown.org }).}
\and Yunhui He\thanks{Department of Mathematics and Statistics, Memorial University of Newfoundland, St. John's, NL A1C 5S7, Canada
  (\email{yunhui.he@mun.ca}, \email{smaclachlan@mun.ca}).}
  \and Scott MacLachlan\footnotemark[3]}

\usepackage{amsopn}

\ifpdf
\hypersetup{
  pdftitle={LFA of BDDC-like Algorithms},
  pdfauthor={Jed Brown, Yunhui He and Scott MacLachlan}
}
\fi

\begin{document}
\maketitle

\begin{abstract}
Local Fourier analysis is a commonly used tool for the analysis of
multigrid and other multilevel algorithms, providing both insight into
observed convergence rates and predictive analysis of the performance
of many algorithms.  In this paper, for the first time, we adapt
local Fourier analysis to examine variants of two- and three-level
balancing domain decomposition by constraints (BDDC) algorithms, to better understand the eigenvalue distributions and condition number bounds on these preconditioned operators. This adaptation is based on a modified choice of basis for the space of Fourier harmonics that greatly simplifies the application of local Fourier analysis in this setting. The local Fourier analysis is validated by considering the two dimensional Laplacian and predicting the condition numbers of the preconditioned operators with different sizes of subdomains. Several variants are analyzed, showing the two- and three-level performance of the ``lumped'' variant can be greatly improved when used in multiplicative combination with a weighted diagonal scaling preconditioner, with weight optimized through the use of LFA.
\end{abstract}

\begin{keywords}
Balancing domain decomposition by constraints (BDDC), Domain decomposition, Local Fourier analysis, Multiplicative methods
\end{keywords}

\begin{AMS}
65N22, 65N55, 65F08
\end{AMS}

\section{Introduction}
Domain decomposition methods are well-studied approaches for the numerical solution of  partial differential equations
both experimentally and theoretically \cite{MR1090464,MR3450068,MR1064335,MR2104179}, due to their efficiency and robustness for many  large-scale problems, and the need for parallel algorithms. Among the many families of domain decomposition algorithms are  Neumann-Neumann \cite{MR2104179}, FETI-DP \cite{dryja2002feti}, Schwarz \cite{MR1064335,MR2104179}, and Optimized Schwarz \cite{MR3450068,MR1924414}.  Balancing domain decomposition by constraints (BDDC) is one family of non-overlapping domain decomposition methods. While BDDC was first introduced by Dohrmann in \cite{dohrmann2003preconditioner}, several variants  have recently been proposed. BDDC-like methods have been successfully applied to many PDEs, including elliptic problems \cite{li2006feti}, the incompressible Stokes equations \cite{li2006bddc,MR2334130}, H(curl) problems \cite{dohrmann2016bddc}, flow in porous media \cite{tu2006bddc}, and the incompressible elasticity problem \cite{MR2334091,pavarino2010bddc}. Theoretical analysis of BDDC has primarily been based on finite-element approximation theory \cite{brenner2007bddc,MR2372024,mandel2003convergence,mandel2005algebraic}. It has been shown that the condition number of the preconditioned BDDC operator can be bounded by  a function of $\frac{H}{h}$(where $h$ is the meshsize, and $H$ is the subdomain size), independent of  the number of subdomains  \cite{mandel2003convergence,mandel2005algebraic}.  A nonoverlapping domain decomposition method for discontinuous Galerkin based on the BDDC algorithm is presented in \cite{MR3666858}, and the condition number of the preconditioned system is shown to be bounded by similar estimates
as those for conforming finite element methods.   BDDC methods in three- or multilevel forms  have also been developed \cite{MR2457352,MR2341811,MR2282536}, and good implementations are available, for example, \cite{badia2018fempar,sousedik2013adaptive,zampini2016pcbddc}.

Since BDDC algorithms are  widely used to solve many problems with high efficiency and parallelism, better understanding of how this methodology works is useful in the design of new algorithms. Local Fourier analysis (LFA), first introduced by Brandt \cite{MR0431719} and well-studied for multigrid methods \cite{briggs2000multigrid,stuben1982multigrid,MR1807961,MR1156079,wienands2004practical}, is an analysis framework that provides predictive performance estimates for many multilevel iterations and preconditioners. Early application of LFA was mainly focused on scalar problems or systems of PDEs with collocated discretizations. Standard LFA \cite{MR1807961} cannot be directly applied to higher-order and staggered-grid discretizations, since the analysis  depends on Toeplitz operator structure inherited from the mesh and  discretization. There is a long history of generalization of ``standard'' LFA to account for more general structure of the discrete operator and/or multigrid algorithm, beginning with analysis of red-black (and, later, multicolour) relaxation \cite{kuo1989two,rodrigo2012multicolor,MR1807961}. This work has been generalized in recent years, leading to the idea of ``periodic stencils'' \cite{bolten2018fourier} and similar approaches for PDEs with random coefficients \cite{kumar2018cell}. From a different perspective, similar tools have arisen to account for the structure of coupled systems of PDEs, beginning with \cite{boonen2008local}. This work was expanded by MacLachlan and Oosterlee  \cite{maclachlan2011local}, to account for the structure of overlapping relaxation schemes for the Stokes Equations. In a third setting, the mode analysis of certain space-time multigrid methods, Friedhoff and MacLachlan again use a similar approach to handle coarsening in time  \cite{friedhoff2015generalized,friedhoff2013local}. In this work, we show how  a generalization of these approaches can be applied to the structure of domain decomposition algorithms.

To our knowledge, there has been no research applying local Fourier analysis to  BDDC-like algorithms. The same is true of the closely related finite element tearing and interconnecting dual-primal (FETI-DP) methodology \cite{dryja2002feti,farhat2001feti,li2005dual}.  Here, we adapt LFA to this domain decomposition method by borrowing tools from LFA for systems of PDEs and other block-structured settings.  Noting that the stencil for the domain decomposition method depends on the size of the subdomains, an adaptation of the standard basis is useful, which we present in Section \ref{intr-new-basis}. Because LFA can reflect both the distribution of eigenvalues and associated eigenvectors of a preconditioned operator, here, we adopt LFA to analyze variants of  the common ``lumped'' and ``Dirichlet'' BDDC algorithms, based on \cite{li2007use}, to guide construction of these methods. To do this, we use a modified basis as in \cite{bolten2018fourier,kumar2018cell,maclachlan2011local} for the Fourier analysis that is well-suited for application to domain decomposition preconditioners.

 Applying the two-level BDDC algorithm requires the solution of a Schur complement equation (coarse problem), which usually poses some difficulty with  increasing problem size. Two- and three-level variants are, thus, considered  in this paper. However, as is well-known in the literature, bounds on the performance of BDDC degrade sharply from two-level to three-level methods, particularly for large values of $H/h$. Since our analysis shows that  the largest eigenvalues of the preconditioned operator for the lumped BDDC algorithm are associated with oscillatory  modes, we propose variants of standard BDDC based on  multiplicative preconditioning and multigrid ideas. From the condition numbers offered by LFA, we can easily compare the efficiency of these variants. Furthermore, LFA can provide optimal parameters for these multiplicative methods, helping tune and understand sensitivity to the parameter choice.

This paper is organized as follows. In Section \ref{sec:Dis-Scheme}, we introduce the finite element discretization of the Laplace problem in two dimensions and the lumped and Dirichlet preconditioners. Two- and three-level preconditioned operators are developed in Section \ref{sec:23-level-variants}.  In Section \ref{sec:LFA-analysis}, we discuss  the Fourier representation of the preconditioned operators.  Section \ref{sec:Numer} reports LFA-predicted condition numbers of the BDDC variants considered here. Conclusions are presented in Section \ref{sec:concl}.

\section{Discretization}\label{sec:Dis-Scheme}
We consider the two-dimensional Laplace problem in weak form:  Find $u\in H^1_0(\Omega):=V$ such that
\begin{equation}\label{Lap-problem}
a(u,v)=\int_{\Omega}\nabla u \cdot \nabla vd\Omega=\langle f,v\rangle, \forall v \in V,
\end{equation}
where $\Omega\subset\mathbb{R}^2$ is a bounded domain with
Lipschitz boundary $\partial\Omega$. Here, we consider the Ritz-Galerkin approximation over $V_h$, the space of piecewise bilinear functions on a uniform rectangular mesh of $\Omega=[0,1]^2$. The corresponding linear system of equations is given as
\begin{equation}\label{Global-prob}
  A x  =b.
\end{equation}
We partition the domain, $\Omega$, into $N$ nonoverlapping subdomains, $\Omega_i, i=1,2,\cdots,N$, where each subdomain is a union of shape regular
elements and the nodes on the boundaries of neighboring subdomains match across the interface $\Gamma=\bigcup \partial \Omega_i\backslash \partial \Omega$. The interface of subdomain $\Omega_i$ is defined by $\Gamma_i=\partial\Omega_i\bigcap\Gamma$. Here, we  consider $\Omega$ with both a discretization mesh (with meshsize $h$) and subdomain mesh (with meshsize $H=ph$) given by uniform grids with square elements or subdomains.

The finite-element space $V_h$ can be rewritten as  $V_{h} = V_{{\rm I},h} \bigoplus V_{\Gamma,h},$
where $V_{{\rm I},h}$ is the product of the subdomain interior variable
spaces $V_{{\rm I},h}^{(i)}$. Functions in $V_{{\rm I},h}^{(i)}$ are supported in the subdomain
$\Omega_i$ and vanish on the subdomain interface $\Gamma_{i}$.  $V_{\Gamma,h}$ is
the space of traces on $\Gamma$ of functions in $V_{h}$.  Then, we can write the subdomain problem with Neumann boundary conditions  on $\Gamma_i$ as
\begin{equation}\label{subproblem}
   A^{(i)}x^{(i)}=  \begin{pmatrix}
      A_{{\rm II}}^{(i)} & A_{\Gamma {\rm I}}^{(i)^{T}}\\
     A_{\Gamma {\rm I}}^{(i)} & A_{\Gamma\Gamma}^{(i)}\\
    \end{pmatrix}
    \begin{pmatrix} x_{{\rm I}}^{(i)} \\ x_{\Gamma}^{(i)}\end{pmatrix}
  =\begin{pmatrix} b_{{\rm I}}^{(i)} \\ b_{\Gamma}^{(i)}\end{pmatrix},
\end{equation}
where $x^{(i)}=( x_{{\rm I}}^{(i)}, x_{\Gamma}^{(i)})\in V^{(i)}_{h}=(V^{(i)}_{{\rm I},h},V^{(i)}_{\Gamma,h})$, and $T$ is the (conjugate) transpose.
Then, the global problem (\ref{Global-prob}) can be assembled from the subdomain problems (\ref{subproblem}) as
\begin{equation*}
  A = \sum_{i=1}^{N} R^{(i)^{T}} A^{(i)} R^{(i)},\,\, \text{and}\,\, b = \sum_{i=1}^{N}R^{(i)^{T}}b^{(i)},
\end{equation*}
where $R^{(i)}$ is the restriction operator from a global vector to a subdomain vector on $\Omega_i$.


\subsection{A Partially Subassembled Problem}
In order to describe variants of the BDDC methods, we first introduce a partially subassembled problem, following \cite{li2007use}, and the corresponding space of partially subassembled variables,
\begin{equation}\label{partial-space}
  \hat{V}_h = V_{\Pi,h} \bigoplus V_{r,h} ,
\end{equation}
where $V_{\Pi,h}$ is spanned by the subdomain vertex nodal basis functions (the coarse degrees of freedom). The complementary space, $V_{r,h}$, is the product of
the subdomain spaces $V_{r,h}^{(i)}$, which correspond to the subdomain
interior and  interface degrees of freedom and  are spanned by the basis functions which vanish
at the coarse-grid degrees of freedom. For a $4\times 4$ mesh, the degrees of freedom in $V_{\Pi,h}$ are those corresponding to the circled nodes at the left of Figure \ref{p4-A}, while the degrees of freedom in $V_{r,h}$ correspond to all interior nodes, plus duplicated (broken) degrees of freedom along subdomain boundaries.

 \begin{figure}
\centering
\includegraphics[width=6.25cm,height=5.5cm]{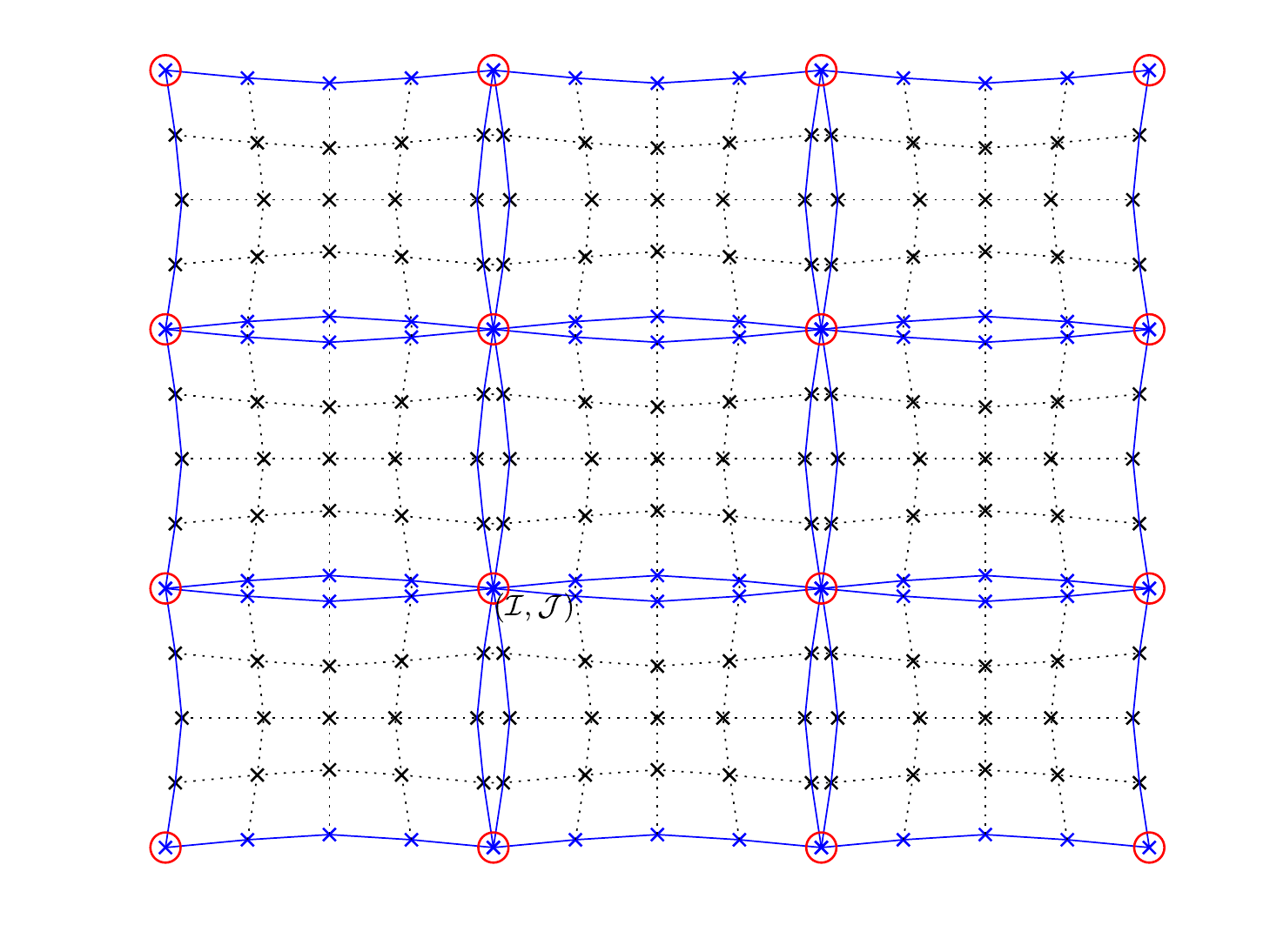}
\includegraphics[width=6.25cm,height=5.5cm]{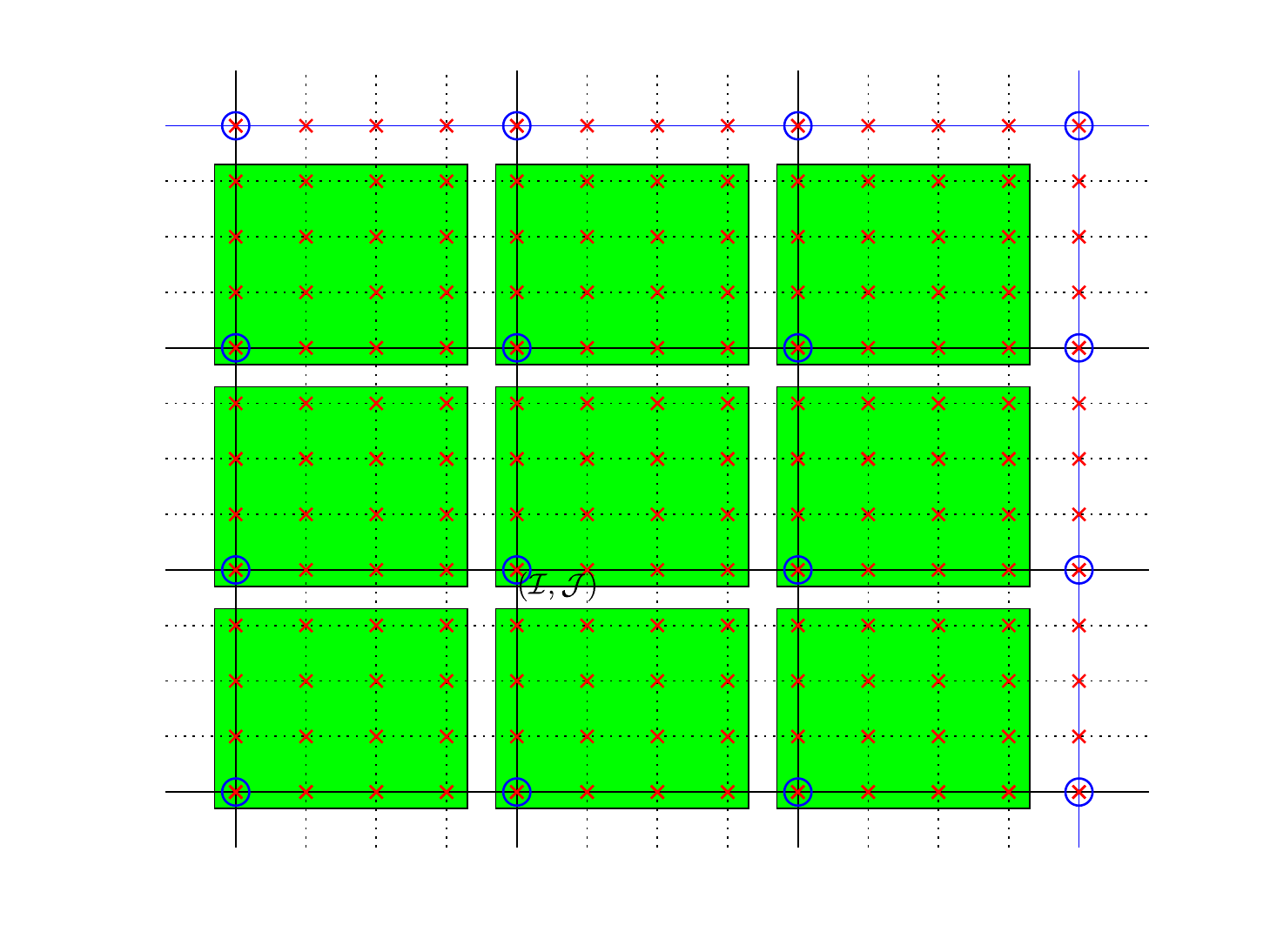}
\caption{At left, the partially broken decomposition given in Equation (\ref{partial-space}), with circled degrees of freedom corresponding to $V_{\Pi,h}$ and all others corresponding to $V_{r,h}$.  This matches the periodic array of subdomains induced by the subsets $\mathfrak{S}_{\mathcal{I},\mathcal{J}}^{*}$ introduced in Equation (\ref{def-SIJ-star}) for $p=4$. At right, a non-overlapping decomposition into subdomains of size $p\times p$ for $p=4$, corresponding to the subsets $\mathfrak{S}_{\mathcal{I},\mathcal{J}}$ introduced in Equation (\ref{def-SIJ}), where LFA works on an infinite
grid and characterizes operators by their action in terms of the
non-overlapping partition denoted by the shading. } \label{p4-A}
\end{figure}

The partially subassembled problem matrix, corresponding
to the variables in the space $\hat{{V}}_{h}$, is obtained by
assembling the subdomain matrices (\ref{subproblem}) only with respect
to the coarse-level variables; that is,
\begin{equation}\label{Partial-Subprob}
  \hat{A} = \sum_{i=1}^{N} (\bar{R}^{(i)})^{T}A^{(i)}\bar{R}^{(i)},
\end{equation}
where $\bar{R}^{(i)}$ is a restriction from space $\hat{V}_{h}$ to $V^{(i)}_{h}$. The central idea behind BDDC preconditioners is to use $\hat{A}$ as an approximation to  $A$ that is more readily inverted, since the subdomain problems can be ``condensed'' to a solve on the underlying coarse grid, as described below.

\begin{remark}
\revise{ The above ``broken'' space can, in fact, be viewed by considering the discretization of the PDE independently on each subdomain and only ``assembling'' the DOFs at the corners of the subdomains. Then, looking at this ``subassembled''  problem, the interface between any two subdomains has independent DOFs associated with each  subdomain at the same locations, which are normally assembled to formulate the global system. This is different than the standard overlapping domain decomposition, which uses only  one DOF at these locations in the overlap between subdomains and contributes this value to each of the shared subdomains.  While $A$ can be readily defined in terms of $\hat A$, the reverse is not true.
}
\end{remark}

\subsection{ Lumped and Dirichlet Preconditioners}
In order to define the  preconditioners under consideration for (\ref{Global-prob}),  we
introduce a positive scaling factor, $\delta_i(\boldsymbol{ x})$,  for each node $\boldsymbol{\boldsymbol {x}}$
on the interface $\Gamma_i$ of subdomain $\Omega_i.$  Let $\mathcal{N}_{\boldsymbol{ x}}$ be the set of indices
of the subdomains that have $\boldsymbol{ x}$ on their boundaries. Define $\delta_i(\boldsymbol{ x})=1/|\mathcal{N}_{\boldsymbol{ x}}|$, where $|\mathcal{N}_{\boldsymbol{ x}}|$ is the cardinality of $\mathcal{N}_{\boldsymbol{ x}}$. The scaled injection operator, $\mathcal{R}_1$, is defined so that each column of $\mathcal{R}_1$ corresponds to a degree of freedom
of the global problem (\ref{Global-prob}). For subdomain interior and coarse-level variables, the corresponding column of $\mathcal{R}_1$ has a single entry with value 1. Columns that correspond to an interface degree of freedom $\boldsymbol{ x}\in \Gamma_{i}$ (the set of nodes in $\Gamma_i$) have $|\mathcal{N}_{\boldsymbol{ x}}|$ non-zero entries each of $\delta_i(\boldsymbol {x})$.

Based on the partially subassembled problem, the first preconditioner introduced  for solving (\ref{Global-prob}) is
\begin{equation*}\label{BDDC-pre}
  M_1^{-1}= \mathcal{R}_{1}^{T}\hat{A}^{-1}\mathcal{R}_1.
\end{equation*}
The preconditioned operator $M_1^{-1}A$ has the same eigenvalues as the preconditioned FETI-DP operator with a lumped preconditioner, except for some eigenvalues equal to 0 and 1 \cite{farhat2001feti,li2007use}. We refer to $M_1$ as the {\it lumped} preconditioner. Note that $M_1$ corresponds to a naive approximation of $A$ by $\hat{A}$, using equal weighting of residual and corrections to/from the duplicated degrees of freedom in the partially broken space.

A similar  preconditioner for $A$ augments this by using discrete harmonic extensions in the restriction and interpolation operators \cite{li2007use}, giving
\begin{equation}\label{M2-Pre}
  M_2^{-1} = (\mathcal{R}_1^{T}-\mathcal{H}J_D) \hat{A}^{-1}\underbrace{(\mathcal{R}_1-J_D^{T}\mathcal{H}^{T})}_{:=\mathcal{R}_2 },
\end{equation}
where $\mathcal{H}$ is the direct sum of local operators $\mathcal{H}^{(i)}=-(A_{{\rm II}}^{(i)})^{-1} (A_{\Gamma {\rm I}}^{(i)})^{T}$, which map the jump over a subdomain interface (given by $J_D$) to the interior of the subdomain by solving a local Dirichlet problem, and gives zero for other values. For any given $v\in \hat{V}_h$, the component of $J_D^{T}v$ on subdomain $\Omega_i$ is given by
\begin{equation}\label{JD-define}
\big(J_D^{T}v(\boldsymbol{x})\big)^{(i)} = \sum_{j\in \mathcal{N}_{\boldsymbol{x}}}\left(\delta_j(\boldsymbol{x})v^{(i)}(\boldsymbol{x})-\delta_i(\boldsymbol{x}) v^{(j)}(\boldsymbol{x})\right), \,\forall \boldsymbol{x}\in\Gamma_{i}.
\end{equation}
 Extending the interface values using the discrete harmonic extension minimizes the energy norm  of the resulting vector \cite{MR2104179}, giving a better stability bound. Furthermore, the preconditioned operator $M_2^{-1}A$ has the same eigenvalues as the BDDC operator \cite{li2006feti}, except for some eigenvalues equal to 1 \cite{li2007use}. We refer to $M_2$ as the {\it Dirichlet} preconditioner. Note that $M_2$ differs from $M_1$ only by local operators in the operators used to map between the fully assembled and partially broken spaces.

Standard bounds (see, e.g., \cite{li2007use}) on  the condition numbers of the preconditioned operators are that, for $M_1^{-1}A$, there exists $\mathfrak{C}_{1,0}\geq 0$ such that  $\kappa \leq \mathfrak{C}_{1,0}\frac{H}{h}(1+{\rm log}\frac{H}{h})$ and, for $M_2^{-1}A$, there exists $\mathfrak{C}_{2,0}\geq 0$ such that  $\kappa \leq \mathfrak{C}_{2,0}(1+{\rm log}\frac{H}{h})^2$.

\begin{remark}
\revise{While the preconditioners given by $M_i^{-1}$ have some notational similarity to Galerkin corrections, it is important to note that $\hat{A}$ acts on a higher-dimensional space than $A$, and does not satisfy $\hat{A}=\mathcal{R}_iA\mathcal{R}_i^T$ for either ``restriction'' operator. The essential difference between the two preconditioners is in how $\mathcal{R}_i$ ``glues'' the solution over the broken space back into the usual continuous space. $\mathcal{R}_1$ does this through a simple partition of unity, while $\mathcal{R}_2$ \revise{minimizes subdomain energy when} propagating mismatches along the subdomain boundaries.}
\end{remark}

\section{Two- and Three-level Variants}\label{sec:23-level-variants}
In both of the above preconditioned operators, we need to solve the partially subassembled problems,  that we write in block form
\begin{equation}\label{partial-problem}
   \underbrace{\begin{pmatrix}
      A_{rr} & \hat{A}_{\Pi r}^{T}\\
     \hat{A}_{ \Pi r} & A_{\Pi\Pi}
    \end{pmatrix}}_{\hat A}
    \underbrace{\begin{pmatrix} \hat{x}_{r} \\ \hat{x}_{\Pi}\end{pmatrix}}_{\hat x}
    = \begin{pmatrix}
      A_{rr} & 0\\
     \hat{A}_{ \Pi r} & \hat{S}_{\Pi}
    \end{pmatrix}
     \begin{pmatrix}
      I & A_{rr}^{-1}\hat{A}_{\Pi r}^{T}\\
    0 & I
    \end{pmatrix}
    \begin{pmatrix} \hat{x}_{r} \\ \hat{x}_{\Pi}\end{pmatrix}
    = \underbrace{\begin{pmatrix} \hat{d}_{r} \\ \hat{d}_{\Pi}\end{pmatrix}}_{\hat{d}},
\end{equation}
where $I$ is the identity matrix, $\hat{S}_{\Pi}=A_{\Pi\Pi}-\hat{A}_{ \Pi r} A_{rr}^{-1}\hat{A}_{ \Pi r}^{T}$ and $\hat{x}_{r}$ contains the subdomain interior and interface degrees of freedom, and $\hat{x}_{\Pi}$ corresponds to the coarse-level degrees of freedom, which are located at the corners of the subdomains. We write $\hat{A}$ in (\ref{partial-problem}) in factorization form to easily separate the action on the coarse degrees of freedom, and to find the corresponding symbol of $\hat{A}^{-1}$. If we define
\begin{equation*}
 P= \begin{pmatrix}
  -A_{rr}^{-1}\hat{A}_{\Pi r}^{T}\\
  I
  \end{pmatrix},
\end{equation*}
then the Schur complement is  simply the Galerkin coarse operator, $\hat{S}_{\Pi}=P^{T}\hat{A}P$, and the block-factorization solve for $\hat{A}$  is  equivalent to a two-level additive multigrid method with exact $F$-relaxation using
\begin{equation*}
   S_{F}= \begin{pmatrix}
  A_{rr}^{-1}& 0\\
  0 & 0
  \end{pmatrix},
\end{equation*}
to define an ``ideal'' relaxation scheme to complement the coarse-grid correction defined by $P$.

In the partially subassembled problem (\ref{partial-problem}), we need to solve a coarse problem related to $\hat{S}_{\Pi}$. We can either solve this coarse problem exactly (corresponding to a two-level method, where the Schur complement is inverted exactly) or inexactly (as a three-level method), where the lumped and Dirichlet preconditioners defined above are used recursively to solve this problem.

\subsection{Exact and Inexact Solve for the Schur Complement}
 Let
\begin{equation*}
  \mathcal{K}_{LD} = \begin{pmatrix}
      A_{rr} & 0\\
     \hat{A}_{ \Pi r} & \hat{S}_{\Pi}
    \end{pmatrix},\,\
    \mathcal{K}_{U}= \begin{pmatrix}
      I & A_{rr}^{-1}\hat{A}_{\Pi r}^{T}\\
    0 & I
    \end{pmatrix},
\end{equation*}
and note that  $\mathcal{K}_{LD}$ and $\mathcal{K}_{U}$ are formed
from the LDU factorization of
$\hat{A}$ as in (\ref{partial-problem}). For $i=1,2,j=0,1,2$, let
$G_{i,j}$ denote the preconditioned operators for  two- and
three-level variants of BDDC, where $i$ and $j$ denote using $M_i$ and
$M_{s,j}$ (with $M_{s,0}:=\hat{S}_{\Pi})$ as preconditioners for the
fine and coarse problems, respectively, where $M^{-1}_{s,j}$ stands
for applying the preconditioner $M_j$ to the Schur complement matrix, $\hat{S}_{\Pi}$.  By standard calculation, we can write
\begin{equation*}
  G_{i,j} = \mathcal{R}^{T}_i\mathcal{K}_{U}^{-1}\mathcal{P}_j\mathcal{K}_{LD}^{-1}\mathcal{R}_iA,
\end{equation*}
with
 \begin{equation*}
    \mathcal{P}_j =\begin{pmatrix}
      I & 0\\
    0 &  M^{-1}_{s,j}\hat{S}_{\Pi}
    \end{pmatrix}.
\end{equation*}

\begin{remark}
When $j=0$, $G_{i,j}$ is a two-level method, solving the Schur complement problem exactly, as $\mathcal{P}_0\equiv I$.  Note that, for the three-level variants $(j=1,2)$,
\begin{equation*}
   \mathcal{P}_j\mathcal{K}_{LD}^{-1} =\begin{pmatrix}
      I & 0\\
    0 & M_{s,j}^{-1}\hat{S}_{\Pi}
    \end{pmatrix}
    \begin{pmatrix}
      A_{rr}^{-1} & 0\\
     -\hat{S}_{\Pi}^{-1}\hat{A}_{ \Pi r}A_{rr}^{-1} & \hat{S}_{\Pi}^{-1}
    \end{pmatrix}= \begin{pmatrix}
      A_{rr}^{-1} & 0\\
     -M_{s,j}^{-1}\hat{A}_{ \Pi r}A_{rr}^{-1} & M_{s,j}^{-1}
    \end{pmatrix},
\end{equation*}
corresponding to replacing  $\hat{S}_{\Pi}$ in (\ref{partial-problem}) by $M_{s,j}$ representing the inexact solve. From this, it is clear that $G_{i,j}$ can be applied without directly applying the inverse of $\hat{S}_{\Pi}$.
\end{remark}

{{
\begin{theorem}\label{Gij-bound}
The eigenvalues of $G_{i,j}$ are real and  bounded below by 1.
\end{theorem}
\begin{proof}
This result is well-known \cite{brenner2007bddc,li2006feti} for the two-level methods, $G_{i,0}$.  For $j\neq 0$, note that
  \begin{eqnarray*}
\mathcal{K}_{LD}\mathcal{P}_j^{-1}\mathcal{K}_{U}
                  &=& \begin{pmatrix}
      A_{rr} & 0\\
      \hat{A}_{\Pi} & \hat{S}_{\Pi}
    \end{pmatrix}
    \begin{pmatrix}
      I &0\\
    0 & \hat{S}_{\Pi}^{-1}M_{s,j}
    \end{pmatrix}
    \begin{pmatrix}
      I &A_{rr}^{-1}A_{\Pi,r}^T\\
    0 & I
    \end{pmatrix} \\
   & =&\mathcal{K}_{U}^T\begin{pmatrix}
      A_{rr} &0\\
    0 & M_{s,j}
    \end{pmatrix}\mathcal{K}_{U},
  \end{eqnarray*}
which is symmetric and positive definite (SPD), ensuring the eigenvalues of $G_{i,j}$ are real.
  Similarly, $\hat{A}$ can be rewritten as
    \begin{eqnarray*}
    \hat{A} &=& \begin{pmatrix}
      I & 0\\
      \hat{A}_{\Pi r} A_{rr}^{-1} & I
    \end{pmatrix}
    \begin{pmatrix}
      A_{rr} &0\\
    0 & \hat{S}_{\Pi}
    \end{pmatrix}
    \begin{pmatrix}
      I & A_{rr}^{-1}\hat{A}_{\Pi r}^{T}\\
    0 & I
    \end{pmatrix}=\mathcal{K}_{U}^T\begin{pmatrix}
      A_{rr} &0\\
    0 & \hat{S}_{\Pi}
    \end{pmatrix}\mathcal{K}_{U}.
  \end{eqnarray*}
Since $\hat{A}$ is SPD, so is $\hat{S}_{\Pi}$. We form $M_{s,j}$ from
$\hat{S}_{\Pi}$ in the same way. As for the two-level case
\cite{brenner2007bddc,li2006feti}, we can bound
\begin{equation*}
 u_r^TM_{s,j}u_r \leq u_r^T\hat{S}_{\Pi}u_r, \forall u_r.
\end{equation*}
Writing $u=(u_I,u_r)$, we have
\begin{equation*}
 u_I^T A_{rr}u_I+ u_r^TM_{s,j}u_r\leq u_I^T A_{rr}u_I+ u_r^T\hat{S}_{\Pi}u_r,
\end{equation*}
which leads to
\begin{equation}\label{three-ineq}
 u^T \begin{pmatrix}
      A_{rr} & 0\\
    0 &  M_{s,j}
    \end{pmatrix}u \leq u^T \begin{pmatrix}
      A_{rr} & 0\\
    0 & \hat{S}_{\Pi}
    \end{pmatrix}u.
\end{equation}
Now, for any $y$, taking $u=\mathcal{K}_{U}y$ in (\ref{three-ineq}), we have
\begin{equation*}
y^T  \mathcal{K}_{U}^T\begin{pmatrix}
      A_{rr} & 0\\
    0 &  M_{s,j}
    \end{pmatrix} \mathcal{K}_{U} y \leq y^T \mathcal{K}_{U}^T \begin{pmatrix}
      A_{rr} & 0\\
    0 & \hat{S}_{\Pi}
    \end{pmatrix}\mathcal{K}_{U} y,
\end{equation*}
which is
\begin{equation*}
 y^T \mathcal{K}_{LD}\mathcal{P}_j^{-1}\mathcal{K}_{U} y\leq y^T\hat{A}y.
\end{equation*}
According to \cite[Theorem 4]{li2007use}, we then have
\begin{equation*}
   \langle u,u\rangle_A \leq \langle G_{i,j}u, u\rangle_A, \forall u,
\end{equation*}
which means that the smallest eigenvalue of $G_{i,j}$ is not less than 1.
  \end{proof}

\begin{remark}
In \cite{li2007use},  $\mathcal{K}_{LD}\mathcal{P}_j^{-1}\mathcal{K}_{U}$ represents one or several multigrid V-cycles
or W-cycles for solving the global partially subassembled problem (\ref{partial-problem}). In our setting, we use the lumped and Dirichlet preconditioners recursively, which also keeps the lower bound, 1, for the inexact solve.
\end{remark}
}
}

\begin{remark}\label{THM-Three-level}
Standard bounds (see, e.g., \cite{li2007use,MR2282536}) on the
condition numbers of the three-level preconditioned operators are that
there exists $\mathfrak{C}_{i,j} \geq 0$ such that $\kappa(G_{i,j}) \leq \mathfrak{C}_{i,j}\Upsilon_i\Upsilon_j$, where $\Upsilon_1=\frac{H}{h}(1+{\rm log}\frac{H}{h})$ and $\Upsilon_2=(1+{\rm log}\frac{H}{h})^2$.
\end{remark}

\subsection{Multiplicative Preconditioners}\label{2multiplicative-operator}
 As we shall see, the bounds above are relatively sharp and  the performance of both preconditioners degrades with subdomain size and number of levels. To attempt to counteract this, we consider multiplicative combinations of these preconditioners with a simple diagonal scaling operator, mimicking the use of weighted Jacobi relaxation in classical multigrid methods.  We use ${G}_{i,j}^{f}$ to denote the multiplicative preconditioned operator based on ${G}_{i,j}$ with diagonal scaling on the fine level. Here,

\begin{equation}\label{Multi-fine-level}
  G_{i,j}^{f} =  G_{i,j}+\omega D^{-1}A(I- G_{i,j}), \,\, i=1,2,\,\,j=0,1,2,
  \end{equation}
where $D$ is the diagonal of $A$ and $\omega$ is a chosen relaxation parameter. Note that $I-G_{i,j}^{f}=(I-\omega D^{-1}A)(I-G_{i,j})$, so $G_{i,j}^f$ represents the multiplicative combination given.
{
{
\begin{theorem}\label{Real-eig-thm}
The eigenvalues of $G_{i,j}^f$ are real.
\end{theorem}
\begin{proof}

 First, recall that $A$ is SPD, so it has a unique SPD matrix square
 root, $A^{\frac{1}{2}}$. Note that
 $A^{\frac{1}{2}}G_{i,j}A^{-\frac{1}{2}}=A^{\frac{1}{2}}\mathcal{R}_i^T(\mathcal{K}_{U}^{-1}\mathcal{P}_j\mathcal{K}_{LD}^{-1})
 \mathcal{R}_iA^{\frac{1}{2}}$.  From Theorem \ref{Gij-bound}, we know
 that  $\mathcal{K}_{U}^{-1}\mathcal{P}_j\mathcal{K}_{LD}^{-1}$ is
 symmetric, so $A^{\frac{1}{2}}G_{i,j}A^{-\frac{1}{2}}$ is symmetric
 as well.

$G_{i,j}^f-I$ is similar to $A^{\frac{1}{2}}(G_{i,j}^f-I)A^{-\frac{1}{2}}$, which can be rewritten as
\begin{eqnarray*}
  A^{\frac{1}{2}}(G_{i,j}^f-I)A^{-\frac{1}{2}} &=& A^{\frac{1}{2}}(I-\omega D^{-1}A)(G_{i,j}-I)A^{-\frac{1}{2}} \\
  &=& (I-\omega A^{\frac{1}{2}}D^{-1}A^{\frac{1}{2}})(A^{\frac{1}{2}}G_{i,j}A^{-\frac{1}{2}}-I)
\end{eqnarray*}

From Theorem \ref{Gij-bound}, we know that the eigenvalues of $G_{i,j}$ are not less than 1. Thus, $A^{\frac{1}{2}}G_{i,j}A^{-\frac{1}{2}}-I$ is symmetric positive semi-definite, and  has a unique symmetric positive semi-definite square root. Using the fact that $\lambda(AB)=\lambda(BA)$, we have
\begin{eqnarray*}
  \lambda(G_{i,j}^f-I) &=& \lambda\Big((I-\omega A^{\frac{1}{2}}D^{-1}A^{\frac{1}{2}})(A^{\frac{1}{2}}G_{i,j}A^{-\frac{1}{2}}-I)\Big)\\
                       &=&\lambda\Big((A^{\frac{1}{2}}G_{i,j}A^{-\frac{1}{2}}-I)^{\frac{1}{2}}(I-\omega A^{\frac{1}{2}}D^{-1}A^{\frac{1}{2}})(A^{\frac{1}{2}}G_{i,j}A^{-\frac{1}{2}}-I)^{\frac{1}{2}}\Big)
\end{eqnarray*}
Note that $(A^{\frac{1}{2}}G_{i,j}A^{-\frac{1}{2}}-I)^{\frac{1}{2}}(I-\omega A^{\frac{1}{2}}D^{-1}A^{\frac{1}{2}})(A^{\frac{1}{2}}G_{i,j}A^{-\frac{1}{2}}-I)^{\frac{1}{2}}$ is symmetric. Thus, the eigenvalues of $G_{i,j}^f-I$ are real and so are those of $G_{i,j}^f$.
\end{proof}

}}


 Another variant is the use of multiplicative preconditioning on the coarse level with a similar diagonal scaling. We use $G_{i,j}^{c}$ to denote the resulting multiplicative preconditioner. Here,

 \begin{equation}\label{Multi-coarse-level}
  G_{i,j}^{c} = \mathcal{R}^{T}_i\mathcal{K}_{U}^{-1}\mathcal{P}_j^{c}\mathcal{K}_{LD}^{-1}\mathcal{R}_iA,\,\,  i,j=1,2,
\end{equation}
where
\begin{equation*}
\mathcal{P}_j^{c} =\begin{pmatrix}
      I & 0\\
    0 &  G_{c,j}
    \end{pmatrix},
\end{equation*}
in which
\begin{equation*}
 G_{c,j} = M^{-1}_{s,j}\hat{S}_{\Pi} + \omega D_s^{-1}\hat{S}_{\Pi}(I-M^{-1}_{s,j}\hat{S}_{\Pi}),
\end{equation*}
where $D_s$ is the diagonal of $\hat{S}_{\Pi}$.

Instead of using a single sweep of Jacobi in $G_{c,j}$, we can consider a symmetrized Jacobi operator $G_{c,j}^s$, where $I-G_{c,j}^s=(I-\omega_1D_s^{-1}\hat{S}_{\Pi})(I-M_{s,j}^{-1}\hat{S}_{\Pi})(I-\omega_2D_s^{-1}\hat{S}_{\Pi})$; that is,
\begin{equation*}
 G_{c,j}^{s} = G_{c,j}+\omega_2(I-G_{c,j})D_s^{-1}\hat{S}_{\Pi},
\end{equation*}
then $G_{i,j}^{c}$ changes to
\begin{equation}
 G_{i,j}^{s,c} =  \mathcal{R}^{T}_i\mathcal{K}_{U}^{-1}\mathcal{P}_j^{s,c}\mathcal{K}_{LD}^{-1}\mathcal{R}_iA,\,\,  i,j=1,2,
\end{equation}
for $\mathcal{P}_j^{s,c}$ defined as $\mathcal{P}_j^{c}$ but with  $G_{c,j}^s$ in the $(2,2)$-block.

Finally, we can also apply the multiplicative operators based on diagonal scaling on both  the fine and coarse levels. We denote this as
\begin{equation}\label{Multi-fine-coarse-level}
  G_{i,j}^{f,c} =  G_{i,j}^{c}+\omega_2 D^{-1}A(I- G_{i,j}^{c}), \,\, i=1,2,\,\,j=1,2,
  \end{equation}
where $D$ is the diagonal of $A$ and $\omega_2$ is a chosen relaxation parameter.
{{
\begin{remark}
The eigenvalues of operators $G_{i,j}^c, G_{i,j}^{s,c}$, and
$G_{i,j}^{f,c}$ are not generally all real.  When $\omega_1=\omega_2$,
$ G_{i,j}^{s,c} $ is  similar to a symmetric matrix and we observe a
real spectrum in numerical results, but not for all cases when
$\omega_1 \neq \omega_2$. In most situations, $G_{i,j}^c, G_{i,j}^{s,c}$ and $G_{i,j}^{f,c}$ have complex eigenvalues.
\end{remark}
}}

 In the following, we focus on analyzing the spectral properties of the above preconditioned operators  by local Fourier analysis \cite{MR1807961}. The main  focus of this work is on the operators $\mathcal{K}_{LD},\mathcal{K}_{U},$ and $\mathcal{P}_j$, because the Fourier representations of other operators are just combinations of these three and some simple additional terms.

\section{Local Fourier Analysis}\label{sec:LFA-analysis}
To apply LFA to the BDDC-like methods proposed here, we first review some terminology of classical LFA.  We consider a two-dimensional infinite uniform grid, $\mathbf{G}_{h}$, with
\begin{equation}\label{Inf-mesh}
  \mathbf{G}_{h}=\big\{\boldsymbol{x}_{i,j}:=(x_i,x_j)=(ih,jh), (i,j)\in \mathbb{Z}^2\big\},
\end{equation}
and Fourier functions $\psi(\boldsymbol{\theta},\boldsymbol{x}_{i,j})= e^{\iota\boldsymbol{\theta}\cdot\boldsymbol{x}_{i,j}/h}$ on $\mathbf{G}_{h}$, where $\iota^2=-1$ and $\boldsymbol{\theta}=(\theta_1,\theta_2)$.  Let $L_h$ be a Toeplitz operator acting on $\mathbf{G}_{h}$  as
\begin{equation*}\label{defi-symbol}
  L_{h}  \overset{\wedge}{=} [s_{\boldsymbol{\kappa}}]_{h} \,\,(\boldsymbol{\kappa}=(\kappa_{1},\kappa_{2})\in \mathbb{Z}^2); \, L_{h}w_{h}(\boldsymbol{x})=\sum_{\boldsymbol{\kappa}\in\boldsymbol{V}}s_{\boldsymbol{\kappa}}w_{h}(\boldsymbol{x}+\boldsymbol{\kappa}h),
  \end{equation*}
with constant coefficients $s_{\boldsymbol{\kappa}}\in \mathbb{R} \,(\textrm{or} \,\,\mathbb{C})$, where $w_{h}(\boldsymbol{x})$ is a function on $\mathbf{G}_{h}$. Here, $\boldsymbol{V}$ is taken to be a finite index set. Note that since $L_h$ is Toeplitz, it is diagonalized by the Fourier modes $\psi(\boldsymbol{\theta},\boldsymbol{x})$.

\begin{definition}\label{formulation-symbol}
 If for all grid functions, $\psi(\boldsymbol{\theta},\boldsymbol{x})$,
 \begin{equation*}
   L_{h}\psi(\boldsymbol{\theta},\boldsymbol{x})= \widetilde{L}_{h} (\boldsymbol{\theta})\psi(\boldsymbol{\theta},\boldsymbol{x}),
 \end{equation*}
 we call $\widetilde{L}_{h}(\boldsymbol{\theta})=\displaystyle\sum_{\boldsymbol{\kappa}\in\boldsymbol{V}}s_{\boldsymbol{\kappa}}e^{\iota\boldsymbol{\theta}\boldsymbol{\kappa}}$ the symbol of $L_{h}$.
\end{definition}
\begin{remark}
  In Definition \ref{formulation-symbol}, the operator $L_{h}$ acts on a single function on $\mathbf{G}_{h}$, so  $\widetilde{L}_{h}$ is a scalar. For an operator mapping vectors on $\mathbf{G}_{h}$ to vectors on $\mathbf{G}_{h}$, the symbol will be extended to be a matrix.
\end{remark}
\subsection{Change of Fourier Basis}\label{intr-new-basis}

Here, we discuss LFA for a domain decomposition method.   While the classical basis set for LFA, denoted $\mathbf{E}_h$ below, could be used, we find it is substantially more convenient to make use of a transformed ``sparse'' basis, introduced here as $\mathbf{F}_H$. This basis allows a natural expression of the periodic structures in domain decomposition preconditioners that vary with subdomain size. We treat each subdomain problem as one macroelement patch, and each subdomain block in the  global problem is diagonalized by a coupled set of Fourier modes introduced in the following. Because each subdomain has the same size, $p\times p$, we consider the high and low frequencies for coarsening by factor $p$, given by
\begin{equation*}
  \boldsymbol{\theta}\in T^{{\rm low}} =\left[-\frac{\pi}{p},\frac{\pi}{p}\right)^{2}, \, \boldsymbol{\theta}\in T^{{\rm high}} =\displaystyle \left[-\frac{\pi}{p},\frac{(2p-1)\pi}{p}\right)^{2} \bigg\backslash \left[-\frac{\pi}{p},\frac{\pi}{p}\right)^{2}.
\end{equation*}
Let $\boldsymbol{\theta}^{(q,r)}=(\theta_1^{(q)},\theta_2^{(r)})$, where $\theta_1^{(q)}=\theta_1^{(0)}+\frac{2\pi q}{p}$ and $\theta_2^{(r)}=\theta_2^{(0)}+\frac{2\pi r}{p}$ for $0\leq q,r <p$.
For any given $\boldsymbol{\theta}^{(0,0)}=(\theta_1^{(0)},\theta_2^{(0)})\in T^{{\rm low}}$,  we define the $p^2$-dimensional space
\begin{equation}\label{Classical-Basis}
  \mathbf{E}_h(\boldsymbol{\theta}^{(0,0)}):={\rm span}\{\psi(\boldsymbol{\theta}^{(q,r)},\boldsymbol{x}_{s,t})= e^{\iota\boldsymbol{\theta}^{(q,r)}\cdot\boldsymbol{x}_{s,t}/h}: q,r=0,1,\cdots,p-1\},
\end{equation}
as the classical space of Fourier harmonics for factor $p$ coarsening.

For any $\boldsymbol{x}_{s,t}\in \mathbf{G}_h$, we consider a grid function defined as a linear combination of the  $p^2$ basis functions for $\mathbf{E}_h(\boldsymbol{\theta}^{(0,0)})$ with frequencies $\{\boldsymbol{\theta}^{(q,r)}\}^{p-1}_{q,r=0}$ and coefficients $\{\beta_{q,r}\}_{q,r=0}^{p-1}$ as
\begin{equation*}
e_{s,t}:=\sum_{q,r=0}^{p-1}\beta_{q,r}\psi(\boldsymbol{\theta}^{(q,r)},\boldsymbol{x}_{s,t}).
 \end{equation*}
We note that any index $(s,t)$ has a unique representation as $(pm+k,pn+\ell)$ where $(m,n) \in \mathbb{Z}^2$ and $k,\ell \in \{0,1,\cdots,p-1\}$.  From (\ref{Classical-Basis}), we have
 \begin{eqnarray}\label{Basis-Transform}
 e_{pm+k,pn+\ell} &=&\sum_{q,r=0}^{p-1}\beta_{q,r} e^{\iota(\theta_1^{(0)}+\frac{2\pi q}{p})x_s/h}e^{\iota(\theta_2^{(0)}
                     +\frac{2\pi r}{p})x_t/h}\nonumber \\
           &=& \sum_{q,r=0}^{p-1}\beta_{q,r}e^{\iota \theta_1^{(0)}x_s/h}e^{\iota \frac{2\pi q(pm+k)}{p}}e^{\iota\theta_2^{(0)}x_t/h}e^{\frac{2\pi r(pn+\ell)}{p}}\nonumber\\
           &=& \sum_{q,r=0}^{p-1}\beta_{q,r} e^{\iota \frac{2\pi qk}{p}}e^{\iota \theta_1^{(0)}x_s/h} e^{\frac{2\pi r\ell}{p}}e^{\iota\theta_2^{(0)}x_t/h}\nonumber\\
           &=& \bigg(\sum_{q,r=0}^{p-1} \beta_{q,r} e^{\iota \frac{2\pi qk}{p}}e^{\frac{2\pi r\ell}{p}}\bigg) \big(e^{\iota \boldsymbol{\theta}^{(0,0)}\cdot\boldsymbol{x}_{s,t}/h}\big).\nonumber
 \end{eqnarray}
Thus, we can write
\begin{equation}\label{DFT-present-1}
        e_{pm+k,pn+\ell}=\hat{\beta}_{k,\ell}e^{\iota \boldsymbol{\theta}\cdot\boldsymbol{x}_{s,t}/H},
        \end{equation}
with
\begin{equation}\label{DFT-present-2}
        \boldsymbol{\theta}=p\boldsymbol{\theta}^{(0,0)},\, \,{\rm and}\,\,\,\hat{\beta}_{k,\ell}=\sum_{q,r=0}^{p-1}\beta_{q,r} e^{\iota \frac{2\pi qk}{p}}e^{\frac{2\pi r\ell}{p}}.
 \end{equation}
In other words,  for any point $(s,t)$ with $\text{mod}(s,p) = k$ and $\text{mod}(t,p) = \ell$, $e_{s,t}$ can be reconstructed from a single Fourier mode with coefficient $\hat{\beta}_{k,\ell}$. Thus, on the mesh $\mathbf{G}_h$ defined in (\ref{Inf-mesh}), the periodicity of the basis functions in $ \mathbf{E}_h(\boldsymbol{\theta}^{(0,0)})$ can also be represented by a pointwise basis on each $p\times p$-block.

Based on (\ref{DFT-present-1}), we consider a ``sparse''  $p^2$-dimensional space as follows
\begin{equation}\label{FG-Sparse-Basis}
  \mathbf{F}_H(\boldsymbol{\theta}):= {\rm span}\{\varphi_{k,\ell}(\boldsymbol{\theta},\boldsymbol{x}_{s,t})= e^{\iota\boldsymbol{\theta}\cdot\boldsymbol{x}_{s,t}/H}\chi_{k,\ell}(\boldsymbol{x}_{s,t}): k,\ell=0,1,\cdots, p-1 \},
\end{equation}
where   $\boldsymbol{\theta}\in[-\pi,\pi)$ and
\begin{equation*}
\chi_{k,\ell}(\boldsymbol{x}_{s,t})=\left\{
 \begin{aligned}
& 1, \quad\,\,\text{if mod}(s,p)=k,\, \text{and}\,\,\text{mod}(t,p)=\ell,\\
& 0,         \quad \,\,\text{otherwise}.
\end{aligned}
\right.
\end{equation*}
Note that, with this notation, (\ref{DFT-present-1}) can be rewritten as
\begin{equation}\label{sparse-basis-presentation}
        e_{pm+k,pn+\ell}=\hat{\beta}_{k,\ell}\varphi_{k,\ell}(\boldsymbol{\theta},\boldsymbol{x}_{s,t}).
      \end{equation}

\begin{theorem}\label{Thm-two-bases}
  $\mathbf{E}_h(\boldsymbol{\theta}^{(0,0)})$  and $\mathbf{F}_H(p\boldsymbol{\theta}^{(0,0)})$ are equivalent.
\end{theorem}

\begin{proof}
While the derivation above shows directly that $\mathbf{E}_h(\boldsymbol{\theta}^{(0,0)})\subset \mathbf{F}_H(p\boldsymbol{\theta}^{(0,0)})$, we revisit this calculation now to show that the mapping $\{\beta_{q,r}\}\rightarrow \{\hat{\beta}_{k,\ell}\}$ is invertible and, hence, $\mathbf{F}_H(p\boldsymbol{\theta}^{(0,0)})\subset \mathbf{E}_h(\boldsymbol{\theta}^{(0,0)})$ as well.

Let $\mathcal{X}$ be an arbitrary vector with size $p^2\times 1$, denoted as
\begin{equation*}
   \mathcal{X}= \begin{pmatrix} \mathcal{X}_{0}\\ \mathcal{X}_{1}\\ \vdots  \\\mathcal{X}_{p-2} \\\mathcal{X}_{p-1}\end{pmatrix},
\,\,{\rm where}\,\,
   \mathcal{X}_{r}=\begin{pmatrix} \beta_{0,r}\\ \beta_{1,r}\\ \vdots  \\\beta_{p-2,r} \\ \beta_{p-1,r}\end{pmatrix},\quad r=0,1,\cdots,p-1.
\end{equation*}
Then, we define a $p^2\times 1$ vector,  $\mathcal{\hat{X}}$, based on (\ref{DFT-present-2}), as follows
\begin{equation*}
   \mathcal{\hat{X}}= \begin{pmatrix} \mathcal{\hat{X}}_{0}\\  \mathcal{\hat{X}}_{1}\\ \vdots  \\ \mathcal{\hat{X}}_{p-2} \\ \mathcal{\hat{X}}_{p-1}\end{pmatrix},
\,\,{\rm where}\,\,
  \mathcal{\hat{X}}_{\ell}=\begin{pmatrix} \hat{\beta}_{0,\ell} \\  \hat{\beta}_{1,\ell} \\  \vdots  \\ \hat{\beta}_{p-2,\ell} \\ \hat{\beta}_{p-1,\ell}\end{pmatrix},\quad \ell=0,1,\cdots,p-1,
\end{equation*}
in which
\begin{equation*}
  \hat{\beta}_{k,\ell} =\sum_{r=0}^{p-1}\Big(\sum_{q=0}^{p-1}\beta_{q,r} e^{\iota \frac{2\pi qk}{p}}\Big)e^{\frac{2\pi r\ell}{p}},\quad q,r=0,1,\cdots,p-1.
\end{equation*}

Let $\mathcal{T}$ be the matrix of this transformation, $\mathcal{\hat{X}}=\mathcal{T}\mathcal{X}$, and
\begin{equation*}
\mathcal{T}_1 =\begin{pmatrix}
(e^{\iota \frac{2\pi}{p}0})^0 & (e^{\iota \frac{2\pi}{p}1})^{0}    &(e^{\iota \frac{2\pi}{p}2})^{0} &\cdots   &(e^{\iota \frac{2\pi}{p}(p-1)})^{0}\\
(e^{\iota \frac{2\pi}{p}0})^1 & (e^{\iota \frac{2\pi}{p}1})^1    &(e^{\iota \frac{2\pi}{p}2})^1 &\cdots   &(e^{\iota \frac{2\pi}{p}(p-1)})^1\\
(e^{\iota \frac{2\pi}{p}0})^{2} & (e^{\iota \frac{2\pi}{p}1})^{2}    &(e^{\iota \frac{2\pi}{p}2})^{2} &\cdots   &(e^{\iota \frac{2\pi}{p}(p-1)})^{2}\\
\vdots  &\vdots   &\vdots  &\vdots &\vdots \\
(e^{\iota \frac{2\pi}{p}0})^{p-2} & (e^{\iota \frac{2\pi}{p}1})^{p-2}    &(e^{\iota \frac{2\pi}{p}2})^{p-2} &\cdots   &(e^{\iota \frac{2\pi}{p}(p-1)})^{p-2}\\
(e^{\iota \frac{2\pi}{p}0})^{p-1} & (e^{\iota \frac{2\pi}{p}1})^{p-1}    &(e^{\iota \frac{2\pi}{p}2})^{p-1} &\cdots   &(e^{\iota \frac{2\pi}{p}(p-1)})^{p-1}
\end{pmatrix}.
\end{equation*}
Note that $\mathcal{T}_1\mathcal{X}_r$ defines a vector whose $(k+1)$-th entry is $\displaystyle\sum_{q=0}^{p-1}\beta_{q,r}e^{2\pi qk/p}$ and, thus, we see that $\mathcal{T} =\mathcal{T}_1 \otimes \mathcal{T}_1$.

Note that $\mathcal{T}_1$ is a $p\times p$ Vandermonde matrix based on values $d_k =e^{\iota \frac{2\pi k}{p}}$, where $k=0,1,2,\cdots,p-1$. It is obvious that $d_j\neq d_k$ if $j\neq k$. Consequently, $\rm det(\mathcal{T}_1)\neq 0$.
Thus, $\mathcal{T}_1$ is invertible, and so is $\mathcal{T}$.
It follows that $\mathbf{E}_h(\boldsymbol{\theta}^{(0,0)})$ and $\mathbf{F}_H(p\boldsymbol{\theta}^{(0,0)})$ are equivalent.
\end{proof}

\begin{remark}
  Let $z=e^{\iota 2\pi/p}$, be the primitive $p$-th root of unity, and note that $(\mathcal{T}_1)_{i,j} =z^{(j-1)(i-1)}$.  Thus, $\widetilde{\mathcal{T}_1}=\frac{1}{\sqrt{p}}\mathcal{T}_1$ is the unitary discrete Fourier transform (DFT) matrix with $\widetilde{\mathcal{T}_1}^{-1} =\widetilde{\mathcal{T}_1}^{T}$,  where $T$ denotes the conjugate transpose. Thus, $\mathcal{T}_1^{-1}=\frac{1}{p}\mathcal{T}_1^{T}$. Similarly, $\mathcal{T}$ is a scaled version of the two-dimensional unitary Fourier transform matrix, and $\mathcal{T}^{-1}=\frac{1}{p^2}\mathcal{T}^{T}$.
\end{remark}

\begin{remark}\label{remark-on-modified-basis}
\revise{In the notation of the proof of Theorem \ref{Thm-two-bases}, the relation between the coefficients, $\{\beta_{q,r}\}$ of a function in $\mathbf{E}_h(\boldsymbol{\theta}^{(0,0)})$, and the coefficients, $\{\hat{\beta}_{k,\ell}\}$, of the same function in the basis of $\mathbf{F}_H(p\boldsymbol{\theta}^{(0,0)})$ is, simply, $\hat{\boldsymbol{\beta}} =T\boldsymbol{\beta}$, where the vectors, $\hat{\boldsymbol{\beta}}$ and $\boldsymbol{\beta}$, are both assumed to follow lexicographic ordering. Thus, if $\widetilde{L}$ is the symbol of a linear operator acting on the basis from $\mathbf{E}_h(\boldsymbol{\theta}^{(0,0)})$, the equivalent symbol in terms of the basis $\mathbf{F}_H(p\boldsymbol{\theta}^{(0,0)})$ is, simply, $\mathcal{T}\widetilde{L}\mathcal{T}^{-1}$.
}
\end{remark}

In the rest of this paper, we use the basis of $\mathbf{F}_H$ as the foundation for local Fourier analysis on the $p\times p$ periodic structures of the BDDC operators. The ``sparse'' (or ``pointwise'') nature of the basis in $\mathbf{F}_H$ allows a natural expression of the operators in BDDC and, as such, is more convenient than the equivalent ``global'' basis in $\mathbf{E}_h$.

Note that the presentation above assumes that the original Fourier space, $\mathbf{E}_h$, is considered with harmonic frequencies in domain $[-\frac{\pi}{p},\frac{(2p-1)\pi)}{p})^2$, and the sparse basis in $\mathbf{F}_H$ considers a single mode, $\boldsymbol{\theta}\in[-\pi,\pi)^2$. In both cases, it is clear that any frequency set covering an interval of length $2\pi$ in both $x$ and $y$ components can be used instead.

\begin{remark}
 The same basis has been used to analyse many other problems, particularly in the case of $p=2$, where it  can be used for red-black Gauss-Seidel relaxation \cite{kuo1989two,MR1807961}. Recent works \cite{bolten2018fourier,kumar2018cell} make use of more general structure similar to what is needed here.
\end{remark}

\subsection{Representation of the Original Problem}\label{sec-ReOP}
On $\mathbf{G}_h$, we call each node, $(\mathcal{I},\mathcal{J})$, where mod$(\mathcal{I},p)=0$ and mod$(\mathcal{J},p)=0$ a coarse-level point index.  We construct a collective grid set associated with $(\mathcal{I},\mathcal{J})$ for each subdomain as
\begin{equation}\label{def-SIJ}
\mathfrak{S}_{\mathcal{I},\mathcal{J}}=\big\{\boldsymbol{ x}_{(\mathcal{I}+k,\mathcal{J}+\ell)}: k, \ell=0,1,\cdots,p-1\big\}.
\end{equation}
In this way, we split the infinite grid $\mathbf{G}_h$ into  $p \times p$ subgrids which coincide with the nonoverlapping partition. Each subdomain can, in fact, be treated as a representative of the overall periodic structure. Similarly, we can define a block stencil on each subdomain. Thus, the degrees of freedom in $A$ can be naturally divided into subsets, $\mathfrak{S}_{\mathcal{I},\mathcal{J}}$, whose union provides a disjoint cover for the set of degrees of freedom on the infinite mesh $\mathbf{G}_{h}$. This division leads naturally to the block operator structure needed for LFA. Throughout the rest of this paper, the index $(\mathcal{I},\mathcal{J})$ corresponds to the coarse point at the lower-left corner of the subdomain under consideration, unless stated otherwise.  The left of Figure \ref{p4-hatA} shows the meshpoints for this decomposition for $p=4$.
 \begin{figure}
\centering
\includegraphics[width=6.25cm,height=5.5cm]{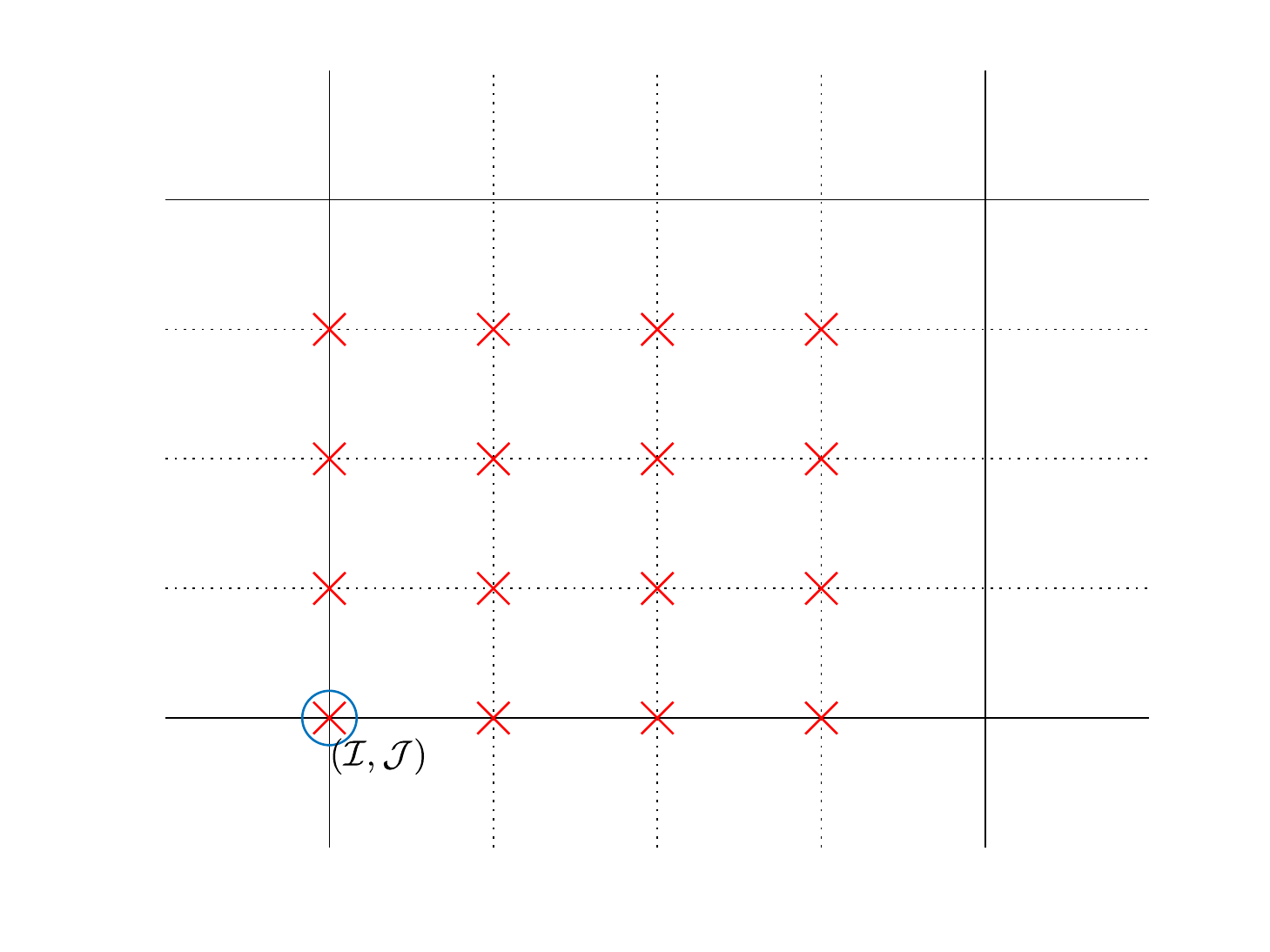}
\includegraphics[width=6.25cm,height=5.5cm]{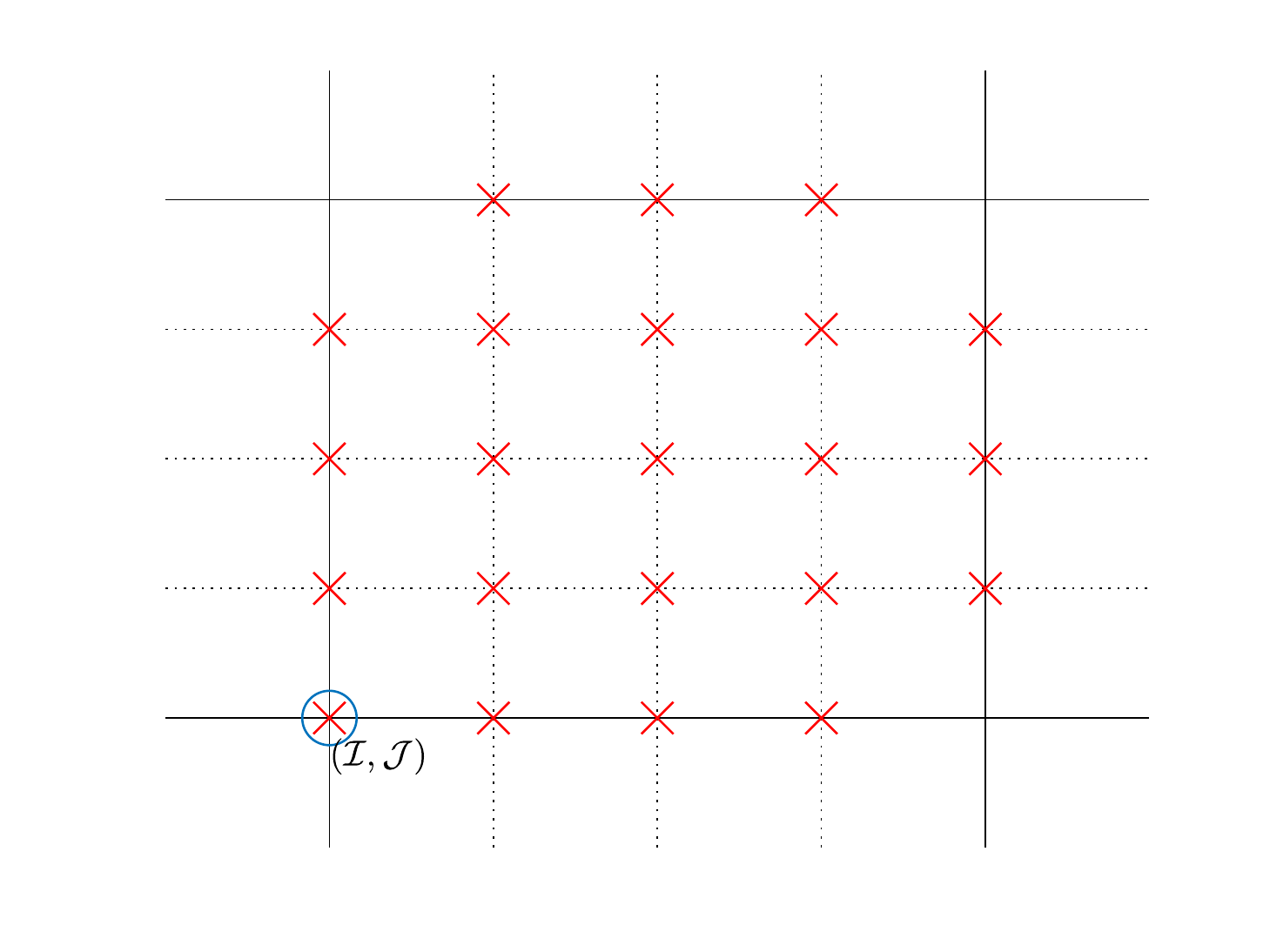}
\caption{At left, the location of degrees of freedom in $\mathfrak{S}_{\mathcal{I},\mathcal{J}}$ defined in Equation (\ref{def-SIJ}) for one subdomain with $p=4$. At right, the location of degrees of freedom in $\mathfrak{S}^{*}_{\mathcal{I},\mathcal{J}}$ defined in Equation (\ref{def-SIJ-star}) for one subdomain with $p=4$.} \label{p4-hatA}
\end{figure}

For each $\mathfrak{S}_{\mathcal{I},\mathcal{J}}$, we use a row-wise ordering of
the grid points (lexicographical ordering). This will fix the ordering of the symbols in the following; for any other ordering, a  permutation operator would need to be applied. In the following, we do not show the specific position of each element in a  vector or matrix, and they are assumed to be consistent with the ordering of the grid points. Based on the set  $\mathfrak{S}_{\mathcal{I},\mathcal{J}}$, we define the $p^2$-dimensional space
\begin{equation}\label{basis-subdomain}
   \mathcal{E}(\boldsymbol{\theta})=\text{span}\big\{\boldsymbol{ \varphi}_{k,\ell}(\boldsymbol{\theta}):   k, \ell =0,1,\cdots, p-1 \big\},
\end{equation}
where $\boldsymbol{ \varphi}_{k,\ell}(\boldsymbol{\theta})=\big(\varphi_{k,\ell}(\boldsymbol{\theta},\boldsymbol{x}_{\mathcal{I}+s,\mathcal{J}+t})\big)_{s,t=0}^{p-1}$ is a $p^2\times 1$ vector with only one nonzero element,  defined in (\ref{FG-Sparse-Basis}), in the position corresponding to $(\mathcal{I}+k,\mathcal{J}+\ell)$. For both $\mathcal{E}(\boldsymbol{\theta})$ and $\boldsymbol{ \varphi}_{k,\ell}(\boldsymbol{\theta})$, we have simply taken the infinite mesh representation of $\mathbf{F}_H$ and truncated it to a single $p\times p$ block of the mesh, which is sufficient to define the symbol of $A$ in this basis.

Note that each  subdomain contains $p^2$ degrees of freedom, and that the corresponding symbol is not a scalar due to the definition of the Fourier basis in (\ref{basis-subdomain}). We treat the block symbol as a system,  presented as a $p^2\times p^2$ matrix, \revise{noting the relation between symbols presented in the classical Fourier basis and those in the basis of (\ref{basis-subdomain}) given  in Remark \ref{remark-on-modified-basis}. To express the symbol of $A$ in terms of this basis, we first find its representation with respect to the classical Fourier modes, then  use the similarity transformation with $\mathcal{T}$ to change basis.   Note that the stencil of the $Q_1$ discretization of the Laplacian is
\begin{equation*}
  L_h = \frac{1}{3}\begin{bmatrix}
  -1 & -1 & -1 \\
  -1 &  8 & -1\\
  -1 & -1& -1
  \end{bmatrix}.
\end{equation*}
 From Definition \ref{formulation-symbol}, the classical symbol of $L_h$ is $\widetilde{L}_h(\boldsymbol{\theta})=\frac{2}{3}\big(4-\cos(\theta_1)-\cos(\theta_2)-2\cos(\theta_1)\cos(\theta_2)\big)$. Thus, the representation of $L_h$ with respect to the traditional basis  in $\mathbf{E}_h$ is a diagonal matrix, denoted as $\widetilde{L}_{p\times p}$, whose diagonal elements are $\widetilde{L}_h(\boldsymbol{\theta}^{(q,r)})$ with $q,r =0,\cdots,p-1$. Thus, the  Fourier representation of  the Laplacian  on $\mathfrak{S}_{\mathcal{I},\mathcal{J}}$  in terms of the basis of $\mathbf{F}_H(p\boldsymbol{\theta}^{(0,0)})$ is
\begin{equation}\label{onesubdomain-symbol-modifid-basis}
\widetilde{A}(p\boldsymbol{\theta}^{(0,0)})= \mathcal{T} \widetilde{L}_{p \times p}\mathcal{T}^{-1}.
\end{equation}
\begin{remark}
We emphasize that in (\ref{onesubdomain-symbol-modifid-basis}), $\widetilde{A}(\boldsymbol{\theta})= \widetilde{A}(p \boldsymbol{\theta}^{(0,0)})$ is a function of $p\boldsymbol{\theta}^{(0,0)}$, where $\boldsymbol{\theta}^{(0,0)} =(\theta_1^{(0)},\theta_2^{(0)})\in T^{\rm low}$.  However, in $\widetilde{L}_{p\times p}$, whose diagonal elements are $\widetilde{L}_h(\boldsymbol{\theta}^{(q,r)})$, we have $\boldsymbol{\theta}^{(q,r)} =\big(\theta_1^{(0)}+\frac{2\pi q}{p}, \theta_2^{(0)}+\frac{2\pi r}{p}\big)\in \big[-\frac{\pi}{p},\frac{(2p-1)\pi}{p}\big)^2$.
\end{remark}
}
The symbol $\widetilde{A}(\boldsymbol{\theta})$ acts  on a vector, $\boldsymbol{\alpha}$, that defines a function in $\mathcal{E}(\boldsymbol{\theta})$  by giving the coefficients of its expansion in terms of the Fourier basis functions. Considering a point in $\mathfrak{S}_{\mathcal{I},\mathcal{J}}$, if the values of a function at neighbouring points are expressed by $\alpha_{k,\ell}\boldsymbol{\varphi}_{k,\ell}$, the entries in $\widetilde{A}(\boldsymbol{\theta})\boldsymbol{\alpha}$ give the coefficients of the Fourier expansion of the function defined by the original operator $A$ on $\mathbf{G}_h$ acting on this function in $\mathcal{E}(\boldsymbol{\theta})$. We note that a similar approach was employed for LFA for vector finite-element discretizations in \cite{maclachlan2011local}.

\begin{remark}
\revise{Note that we can easily use the standard Fourier space, $\mathbf{E}_h$, for the Fourier representation of the Laplacian on $\mathfrak{S}_{\mathcal{I},\mathcal{J}}$. However, the symbols of the preconditioners, $M_i$, are more naturally expressed in terms of the basis in $\mathbf{F}_H$.}
\end{remark}

\subsection{Representation of Preconditioned Operators}\label{sec-RePO}
Now we turn to calculating the Fourier representations of the preconditioners, $M_1^{-1}$ and $M_2^{-1}$. Recall (\ref{partial-space}) and the partially broken decomposition at the left of Figure \ref{p4-A}, where the two DOFs at the boundary of each subdomain are duplicated in the partially broken space (except for the ``coarse'' vertices). When we consider the representation of $\hat{A}$, (\ref{Partial-Subprob}), we must account for this duplication. It is natural to define a bigger collection of DOFs to represent the symbol of this block stencil, compared with the representation of $A$.

First, we define a collective grid set associated with $(\mathcal{I},\mathcal{J})$ for the partially subassembled problem for each subdomain as
\begin{equation}\label{def-SIJ-star}
\mathfrak{S}^{*}_{\mathcal{I},\mathcal{J}}=\{\boldsymbol{ x}_{(\mathcal{I}+k,\mathcal{J}+\ell)}: k, \ell=0,1,\cdots,p\}\setminus \{\boldsymbol{ x}_{(\mathcal{I}+p,\mathcal{J})},\boldsymbol{ x}_{(\mathcal{I},\mathcal{J}+p)}, \boldsymbol{ x}_{(\mathcal{I}+p,\mathcal{J}+p)} \},
\end{equation}
see the right of Figure \ref{p4-hatA}.

Now, we can consider the stencil of $\hat{A}$ acting on one subdomain, $\mathfrak{S}^{*}_{\mathcal{I},\mathcal{J}}$. Let $\hat{A}^{(\mathcal{I},\mathcal{J})}$ be a $(p+1)^2\times (p+1)^2$ matrix, which is the partially subassembled problem on one subdomain including its four neighbouring coarse-grid degrees of freedom, as
\begin{equation}\label{partial-one-element-form}
  \hat{A}^{(\mathcal{I},\mathcal{J})}  =\begin{pmatrix}
       A_{rr}^{(\mathcal{I},\mathcal{J})} & \left(\hat{A}_{ \Pi r}^{(\mathcal{I},\mathcal{J})}\right)^{T} \\
      \hat{A}_{ \Pi r}^{(\mathcal{I},\mathcal{J})} & A_{ \Pi \Pi}^{(\mathcal{I},\mathcal{J})}
    \end{pmatrix}=\begin{pmatrix}
      A_{rr}^{(\mathcal{I},\mathcal{J})} & 0\\
     \hat{A}_{ \Pi r}^{(\mathcal{I},\mathcal{J})} & \hat{S}_{\Pi}^{(\mathcal{I},\mathcal{J})}
    \end{pmatrix}
     \begin{pmatrix}
      I & \left(A_{rr}^{(\mathcal{I},\mathcal{J})}\right)^{-1}\left(\hat{A}_{\Pi r}^{(\mathcal{I},\mathcal{J})}\right)^{T}\\
    0 & I
    \end{pmatrix},
\end{equation}
where $A_{rr}^{(\mathcal{I},\mathcal{J})}$ is a $\big((p+1)^2-4\big)\times \big((p+1)^2-4\big)$ matrix corresponding to the interior and interface degrees of freedom on the subdomain and  $A_{ \Pi \Pi}^{(\mathcal{I},\mathcal{J})}$ corresponds to the four coarse-level variables on one subdomain. Note that $A_{ \Pi \Pi}^{(\mathcal{I},\mathcal{J})}=\frac{2}{3}I$ and  $\hat{S}_{\Pi}^{(\mathcal{I},\mathcal{J})}=A_{ \Pi \Pi}^{(\mathcal{I},\mathcal{J})}-\hat{A}_{\Pi r}^{(\mathcal{I},\mathcal{J})}(A_{rr}^{(\mathcal{I},\mathcal{J})})^{-1}(\hat{A}_{\Pi r}^{(\mathcal{I},\mathcal{J})})^{T}$.   We use index $(\mathcal{I},\mathcal{J})$ as a superscript in order to distinguish this as a subblock of  the matrix  in (\ref{partial-problem}), but note that it is independent of the particular subdomain, $(\mathcal{I},\mathcal{J})$, under consideration. Let  $ \widetilde{\hat{A}}$ be the Fourier representation of the global partially subassembled problem, with the corresponding symbol being a $\big((p+1)^2-3\big)\times \big((p+1)^2-3\big)$ matrix,
\begin{equation*}
  \widetilde{\hat{A}} = \begin{pmatrix}
      \widetilde{ A}_{rr} & 0\\
     \widetilde {A}_{ \Pi r} & \widetilde{S}_{\Pi}
    \end{pmatrix}
     \begin{pmatrix}
      \widetilde {I} & (\widetilde{A}_{rr})^{-1}\widetilde{A}_{\Pi r}^{T}\\
    0 & \widetilde{ I}
    \end{pmatrix}=\widetilde{\mathcal{K}}_{LD} \widetilde{\mathcal{K}}_{U},
\end{equation*}
where $\widetilde{A}_{rr}$ is a $\big((p+1)^2-4\big) \times \big((p+1)^2-4\big)$ Fourier representation of $A_{rr}^{(\mathcal{I},\mathcal{J})}$ computed as was done for $\widetilde{A}$ above and $ \widetilde{S}_{\Pi}$ is the Fourier representation of the global Schur complement, $\hat{S}_{\Pi}$.

In order to compute the Fourier representation of $\hat{S}_{\Pi}$, recall the global block decomposition of $\hat{A}$ in (\ref{partial-problem}) where $\hat{S}_{\Pi}=A_{ \Pi \Pi}-\hat{A}_{\Pi r}A_{rr}^{-1}\hat{A}_{\Pi r}^{T}$. To calculate  $\hat{S}_{\Pi}$, we first calculate  the restriction of $\hat{S}_{\Pi}$ on $\mathfrak{S}^{*}_{\mathcal{I},\mathcal{J}}$, then assemble this to give  the global stencil. Let $S_0=\hat{A}_{\Pi r}^{(\mathcal{I},\mathcal{J})}(A_{rr}^{(\mathcal{I},\mathcal{J})})^{-1}(\hat{A}_{\Pi r}^{(\mathcal{I},\mathcal{J})})^{T}$ be a $4\times 4$  matrix corresponding to the vertices adjacent to one subdomain, representing one macroelement of the coarse-level variables. Direct calculation shows this matrix has the same nonzero structure as the element stiffness matrix for a symmetric second-order differential operator on a uniform square mesh, with equal values for the connections from each node to itself (denoted $s_1$), its adjacent vertices ($s_2$), and its opposite corner ($s_3$). Since $\hat{S}_{\Pi}^{(\mathcal{I},\mathcal{J})}=\frac{2}{3}I-S_0$ gives the macroelement stiffness contribution, assembling the coarse-level stiffness matrix over $2\times 2$ macroelement patches yields $\widetilde{S}_{\Pi}$ as the symbol of the 9-point stencil  given by
\begin{equation*}\label{Schur-stencil}
     \begin{bmatrix}
      -s_3   &   -2s_2              & -s_3\\
     -2s_2   &  \frac{8}{3}-4s_1    & -2s_2  \\
     -s_3    &   -2s_2              &-s_3
    \end{bmatrix},
\end{equation*}
acting on the coarse points.

The final term needed for the symbol of $\hat{A}$ is $\widetilde{A}_{\Pi r}$, the representation of the contribution from interior and interface degrees of freedom to the coarse degrees of  freedom, which has only 12-nonzero elements per subdomain, with 3 contributing to each corner of the subdomain. We take the coarse-level point $\boldsymbol{x}_{I,J}$ as an example. At the right of Figure \ref{p4-hatA}, $\boldsymbol{x}_{\mathcal{I},\mathcal{J}}$ obtains contributions from  the points $\boldsymbol{x}_{\mathcal{I}+1,\mathcal{J}}, \boldsymbol{x}_{\mathcal{I}+1,\mathcal{J}+1},\boldsymbol{x}_{\mathcal{I},\mathcal{J}+1}$ and the corresponding stencils are
\begin{equation*}
\begin{bmatrix}
  * & -\frac{1}{6}
\end{bmatrix}, \,
\begin{bmatrix}
   & -\frac{1}{3}\\
  *&
\end{bmatrix},\,
\begin{bmatrix}
   -\frac{1}{6}\\
   *
\end{bmatrix},
\end{equation*}
where $*$ denotes the position on the grid at which the discrete operator is applied, namely
$\boldsymbol{x}_{\mathcal{I},\mathcal{J}}$. The symbols of these three stencils are given by $-\frac{1}{6}e^{\iota \theta_1/p}, -\frac{1}{3}e^{\iota(\theta_1+\theta_2)/p},$ $-\frac{1}{6}e^{\iota\theta_2/p}$, respectively.
Since $\boldsymbol{x}_{\mathcal{I},\mathcal{J}}$ is adjacent to three other subdomains, the coarse degree of freedom at $\boldsymbol{x}_{\mathcal{I},\mathcal{J}}$ also obtains  contributions from those subdomains,  and the other 9 contributing stencils are computed similarly.

\begin{remark}
\revise{In essence, the symbol of $\hat{A}$ is determined by considering the action of $\hat{A}$ on a function in $\mathbf{F}_H$ that has been mapped into the partially subassembled space by $\mathcal{R}_i$. Such functions have natural periodicity expressed over $\mathfrak{S}_{\mathcal{I},\mathcal{J}}^{*}$, and the only challenge in expressing the symbol of $\hat{A}$ comes from assembling the Schur complement and connections to the ``corner'' (coarse level) DOF, as described above.
}
\end{remark}

For the stencil of $M_1^{-1}$, we need the representation of $\mathcal{R}_1$. Recall that $\mathcal{R}_1$ is a scaling operator, where each column of $\mathcal{R}_1$ corresponding to a degree of freedom of the global problem in the interiors and at the coarse-grid points has a single nonzero entry with value 1, and each column of $\mathcal{R}_1$ corresponding to an interface degree
of freedom has two nonzero entries, each with value $\frac{1}{2}$. Since we consider periodic Fourier modes on each subdomain, the interface degrees of freedom share the same values scaled by an exponential shift. For example, at the left of Figure \ref{p4-hatA}, the degrees of freedom located at the left boundary and the right boundary have the same coefficient of the (shifted) exponential, as do the degrees of freedom located at the  bottom and top. Thus, $\mathcal{R}_1$ is its own Fourier representation, since the neighborhoods do not contribute to each other. Note that $\mathcal{R}_1$ maps the $p^2$-dimensional Fourier basis from $\mathcal{E}(\boldsymbol{\theta})$, used to express $\widetilde{A}(\boldsymbol{\theta})$ onto a $(p+1)^2-3$ dimensional space with similar sparse basis on $\mathfrak{S}^{*}_{\mathcal{I},\mathcal{J}}$ that is used above to express  the symbol of $\hat{A}$ and its inverse.

Finally, the representation of $M_1^{-1}A$ is given by
\begin{equation*}\label{Symbol-invMa}
\widetilde{G}_{1,0}(\boldsymbol{\theta})=\widetilde{ \mathcal{R}}_{1}^{T}(\widetilde{{\hat{A}}})^{-1}\widetilde{\mathcal{R}}_1\widetilde{A}=\widetilde{ \mathcal{R}}_{1}^{T}\widetilde{\mathcal{K}}_{U}^{-1}\widetilde{\mathcal{K}}_{LD}^{-1}\widetilde{\mathcal{R}}_1\widetilde{A}.
\end{equation*}

For the Dirichlet preconditioner in (\ref{M2-Pre}), we also need to know the LFA representation of the operators $J_D$ and $\mathcal{H}$. Since $J_D$ is a pointwise scaling operator, its symbol in the pointwise basis of $\mathbf{F}_H$ is itself.  According to the definition of $\mathcal{H}$, the symbol of $\mathcal{H}$ is given by $\widetilde{\mathcal{H}}=\widetilde{ A}_{rr,{\rm I}}^{-1}\widetilde{A}_{\Gamma,{\rm I}}^{T}$, where $\widetilde{ A}_{rr,{\rm I}}$ is the submatrix of  $\widetilde{ A}_{rr}$ corresponding to the interior degrees of freedom, and  $\widetilde{A}_{\Gamma,{\rm I}}^{T}$ is the submatrix of $\widetilde{\hat{A}}$ corresponding to the contribution of the interface degrees of freedom to the interior degrees of freedom. Both of these are computed in a similar manner to $\widetilde{A}$ and $\widetilde{\hat{A}}$ as described above. Thus, the LFA representation of $M_2^{-1}A$  can be written as
\begin{equation*}\label{invM2A-LFA}
  \widetilde{G}_{2,0}(\boldsymbol{\theta})=(\widetilde{\mathcal{R}}_1^{T}-\widetilde{\mathcal{H}}\widetilde{J}_D)\widetilde{\mathcal{K}}_{U}^{-1}\widetilde{\mathcal{K}}_{LD}^{-1}
  (\widetilde{\mathcal{R}}_1-\widetilde{J}_D^{T}\widetilde{\mathcal{H}}^{T})\widetilde{A}.
\end{equation*}

The details of the 3-level variants of LFA are similar to those given above. We now consider a segment of the infinite mesh given, on the fine level, by a $p\times p$ array of subdomains, with each subdomain of size $p\times p$ elements.  On the first coarse level (corresponding to the Schur complement $\hat{S}_{\Pi}$ in (\ref{partial-problem})), we then consider a single $p\times p$ subdomain of the infinite coarse mesh, and apply the same technique recursively. To accommodate this, we adapt the fine-level Fourier modes to be
$\varphi^{*}(\boldsymbol{\theta},\boldsymbol{x}):=e^{\iota\boldsymbol{\theta}\cdot\boldsymbol{x}/H'}$, where $H'=p^2h$. The coarse-level Fourier modes are then the same as (\ref{basis-subdomain}). Thus, $\widetilde{G_{i,j}}(\boldsymbol{\theta})$ is a $p^4\times p^4$ matrix for the three-level variants.

\begin{remark}
\revise{For practical use of this LFA for BDDC preconditioners of other discretizations or PDEs, we first need to represent the symbol of $A$ in the basis for $\mathbf{F}_H$ over one subdomain, similarly to (\ref{onesubdomain-symbol-modifid-basis}). Then, we need the Fourier representation of the preconditioner on $\mathfrak{S}^{*}_{\mathcal{I},\mathcal{J}}$, where $\mathcal{R}_i$ determines how the Fourier coefficients from $\mathbf{F}_H$  map onto $\mathfrak{S}^{*}_{\mathcal{I},\mathcal{J}}$. Symbols for operators on the partially subassembled space can be calculated as described here for the $Q_1$ Laplacian, along with those of the chosen restriction operator.
}
\end{remark}
\section{Numerical Results}\label{sec:Numer}

\subsection{Condition Numbers of Two-level Variants}
In the LFA  setting, $\boldsymbol{\theta}=(\theta_1,\theta_2)\in[-\pi,
  \pi)^2$. Here we take $d\theta = \pi/n$ as the discrete stepsize and
  sample the Fourier space at $2n$ evenly distributed frequencies in
  $\theta_1$ and $\theta_2$ with offset $\pm d\theta/2$ from
  $\theta_1=\theta_2=0$ to avoid the singularity at zero
  frequency. For each frequency on the mesh, we compute the
  eigenvalues of the two-level operators, and  define
  $\kappa:=\frac{\lambda_{\rm max}}{\lambda_{\rm min}}$, where
  $\lambda_{\rm min}$  and $\lambda_{\rm max}$ are the smallest and
  biggest eigenvalues over all frequencies.  We note that, as proven
  above, the eigenvalues in this setting are always real; moreover, we
  consider only choices of the relaxation parameter, $\omega$, such
  that the eigenvalues are also always positive, so this condition
  number makes sense as a proxy for how ``well preconditioned'' the
  linear system is.

Table \ref{Cond-2-level-G10} shows the condition numbers for the
two-level preconditioners with variation in both subdomain size, $p$,
and sampling frequency, $n$. For comparison, we include a row labelled
PCBDDC of condition numbers estimated by applying the PCBDDC algorithm
from PETSc, as described in Section \ref{sec:validation}.  When $n=2$, the condition number prediction is notably inaccurate, but we obtain a  consistent prediction for $n\geq4$ (and very consistent for $n\geq 8$). For $\widetilde{G}_{1,0}$, the condition number increases quickly with $p$ as expected. Compared with $\widetilde{ G}_{1,0}$, $\widetilde{ G}_{2,0}$ has a much smaller condition number that grows more slowly with $p$. For $\widetilde{G}_{1,0}$, we know there exists $\mathfrak{C}_{1,0}$ such that the true condition number of the preconditioned system (on a finite grid) is bounded by $\mathfrak{C}_{1,0}\frac{H}{h}(1+{\rm log}\frac{H}{h})$ \cite{li2007use}; from this data, we see  that our LFA prediction is consistent with this, with constant $\mathfrak{C}_{1,0}\approx 0.6$. For $\widetilde{G}_{2,0}$, we know there exists $\mathfrak{C}_{2,0}$ such that the true condition number of the preconditioned system (on a finite grid) is bounded by $\mathfrak{C}_{2,0}(1+{\rm log}\frac{H}{h})^2$ \cite{li2007use}; from this data, again we see  that our LFA prediction is consistent with this, with constant $\mathfrak{C}_{2,0}\approx 0.4$.

\begin{table}
\centering
\caption{LFA-predicted condition numbers of two-level preconditioners as
  a function of subdomain size, $p$, and sampling frequency, $n$.  For
  comparison, we include numerical estimates using the procedure
  described in Section \ref{sec:validation}, labelled PCBDDC.}
\begin{tabular}{l||c|c|c|c||c|c|c|c}

     &\multicolumn{3} {c} {$\widetilde{G}_{1,0}$}  && \multicolumn{4} {c} {$\widetilde{G}_{2,0}$} \\
\hline
\backslashbox{$n$}{$p$}   &4    &8                 &16            &32    &4           &8         &16                &32  \\
\hline
$2$                    &4.14      &11.11         &27.95            &67.55      &2.23                  &3.02      &3.94            &5.01        \\

$4$                    &4.36       &11.94         &30.27            &73.44       &2.32                   &3.15      &4.13            &5.26     \\

$8$                    &4.42       &12.18          &30.94            &75.16     &2.34                   &3.19      &4.17            &5.32      \\

$16$                   &4.44      &12.25         &31.12             &75.61           &2.35                   &3.19      &4.19            &5.33       \\
$32$                   &4.44        &12.26         &31.16            &75.72      &2.35            &3.20      &4.19            &5.34          \\

$64$                   &4.44       &12.27        &31.17            &75.75    &2.35                  &3.20      &4.19            &5.34    \\

$128$                  & 4.44       &12.27         &31.18           &75.76        &2.35                   &3.20      &4.19            &5.34  \\
\hline
PCBDDC                 & 4.44 & 12.27 & 31.18 & 75.76 & 2.34 & 3.18 & 4.17 & 5.31 \\
\hline
$\mathfrak{C}_{i,0}(n=32)$   &0.47     &0.50          &0.52             &0.53    &0.41           & 0.34     &0.29            &0.27 \\
\hline
\end{tabular}\label{Cond-2-level-G10}
\end{table}

Optimizing the weight parameters for $\widetilde{G}_{1,0}^{f}$  and $\widetilde{G}_{2,0}^{f}$  by systematic search with different $n$ and $p$, we see that the optimal parameter $\omega$ is dependent on $p$, but largely independent of $n$.  Table \ref{Cond-2-level-G10f} shows that significant improvement can be had for the $M_1$ preconditioner, but not for $M_2$, see Table \ref{Cond-2-level-G20f}. We again see small $n$ (e.g., $n=4$ or 8) is enough to obtain a consistent prediction for these condition numbers.



\begin{table}
\centering
\caption{Condition numbers for two-level lumped preconditioner with fine-grid multiplicative combination with diagonal scaling,   $\widetilde{G}_{1,0}^{f}$. In brackets, value of weight parameter, $\omega$, that minimizes condition number.}
\begin{tabular}{l||c|c|c|c}
\hline
\hline
\backslashbox{$n$}{$p$}   &4     &8      &16       &32 \\
\hline
\hline

$2$  &2.06(2.1)     &3.18(2.3)   &5.43(2.5)   &9.71(2.6)       \\
$4$  &2.17(1.5)     &3.29(2.3)    &5.64(2.5)   &9.99(2.6)  \\
$8$  &2.18(1.4)     &3.32(2.3)   &5.70(2.5)    &10.08(2.6)       \\

$16$  &2.18(1.4)    &3.32(2.3)     &5.72(2.5)  &10.10(2.6)   \\

$32$   &2.18(1.4)   &3.33(2.3)    &5.72(2.5)   &10.10(2.6)    \\

$64$   &2.18(1.4)   &3.33(2.3)    &5.72(2.5)   &10.10(2.6)     \\

\hline
\hline
\end{tabular}\label{Cond-2-level-G10f}
\end{table}

\begin{table}
\centering
\caption{Condition numbers for two-level Dirichlet preconditioner with fine-grid multiplicative combination with diagonal scaling, $\widetilde{G}_{2,0}^{f}$. In brackets, value of weight parameter, $\omega$, that minimizes condition number.}
\begin{tabular}{l||c|c|c|c}
\hline
\hline
\backslashbox{$n$}{$p$}   &4                         &8                 &16               &32 \\
\hline
\hline

$2$                    &1.82(2.2)                 &2.36(1.7)              &3.12(2.0)         &4.20(1.8)     \\
$4$                    &2.03(1.1)                 &2.54(1.6)              &3.33(2.0)         &4.44(1.8)     \\

$8$                    &2.07(1.1)                 &2.59(1.6)              &3.39(2.0)         &4.50(1.8)     \\

$16$                   &2.08(1.1)                 &2.60(1.6)              &3.40(2.0)         &4.52(1.8)     \\

$32$                   &2.08(1.1)                 &2.60(1.6)              &3.40(2.0)         &4.52(1.8)     \\

$64$                   &2.08(1.1)                 &2.61(1.6)              &3.40(2.0)         &4.52(1.8)     \\

\hline
\hline
\end{tabular}\label{Cond-2-level-G20f}
\end{table}

In order to see the sensitivity of performance to parameter choice, we  consider the condition numbers for the two-level lumped  and Dirichlet preconditioners in multiplicative combination with diagonal scaling on the fine grid with $p=4,8,16$ and $n=32$, as a function of $\omega$, in Figure \ref{two-grid-sensitivity}. We see that the condition number of $\widetilde G_{1,0}^{f}$ shows strong sensitivity to small values of $\omega$. For $\widetilde G_{2,0}^{f}$, however,  many allowable parameters obtain a good condition number.

 \begin{figure}
\centering
\includegraphics[width=.8\textwidth]{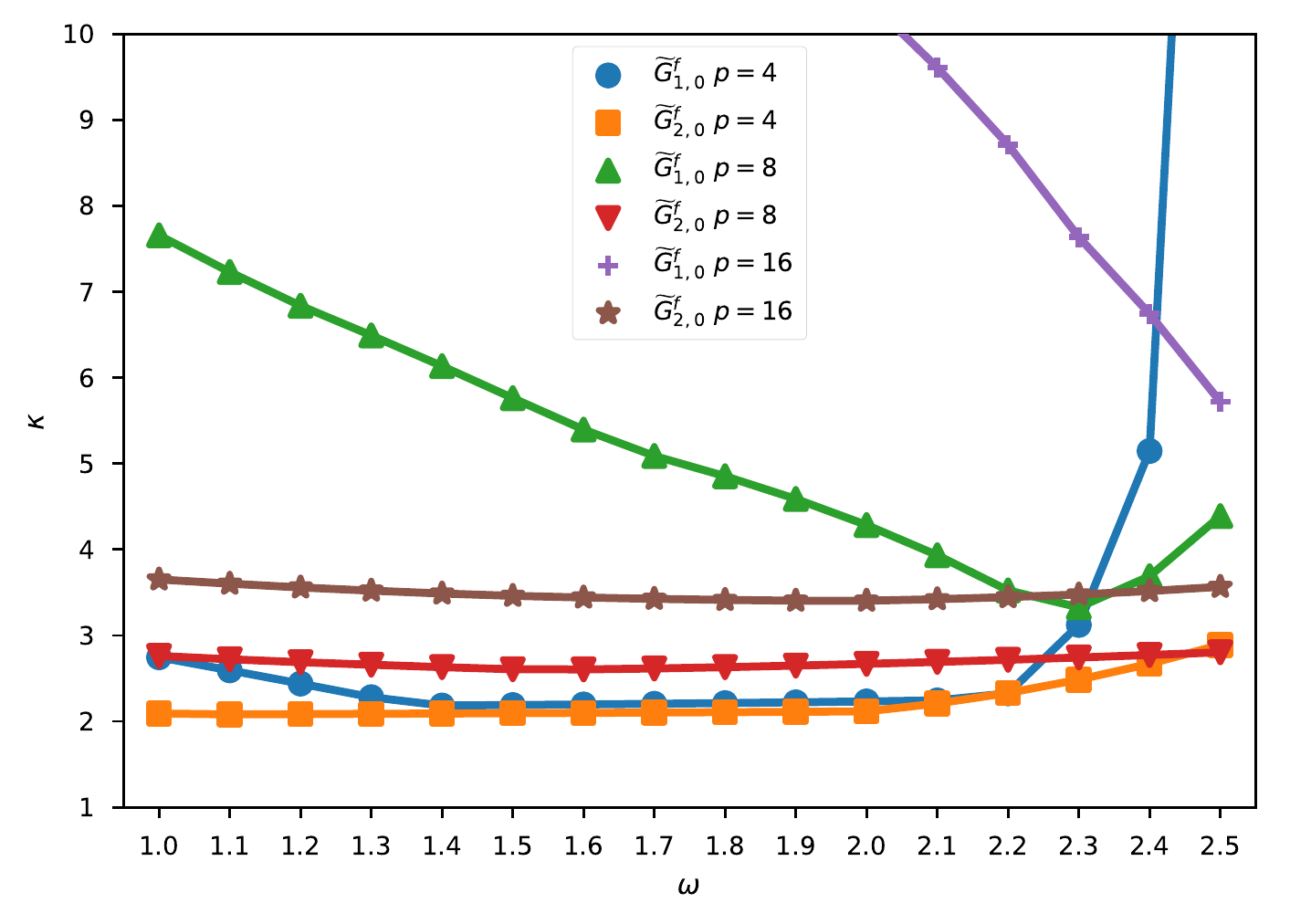}
\caption{Condition numbers of LFA symbols for two-level lumped and Dirichlet preconditioners in multiplicative combination with diagonal scaling on the fine grid with $p\in \{4, 8, 16\}$ and $n=32$, as a function of relaxation parameter, $\omega$.} \label{two-grid-sensitivity}
\end{figure}


 \subsection{Numerical Validation}\label{sec:validation}
{
For validation, numerical results were obtained using PETSc~\cite{petsc-user-ref} version 3.10's PCBDDC~\cite{zampini2016pcbddc} implementation, modified to support lumped variants (these modifications will be available in a future release).  Multiplicative relaxation was performed using the ``composite'' preconditioner type with Richardson/Jacobi.
The example \texttt{src/ksp/ksp/examples/tutorials/ex71.c} was used with periodic boundary conditions and $16\times 16$ subdomains each of size $p\times p$ with periodic boundary conditions.
Use of periodic boundary conditions is significant in that over-relaxation (large $\omega$) requires special treatment at boundaries.
The singular value decomposition was used for the coarse solver as a reliable method for handling the null space of constants, though many other approaches, such as factorization with shifting, can be used in practice and deliver equivalent results.
Eigenvalues of the preconditioned operator were estimated using the Hessenberg matrix computed by solving $A x = 0$ using GMRES with random zero-mean initial guess, no restarts, and modified Gram-Schmidt, converged to a relative tolerance of $10^{-12}$ or 50 iterations, whichever comes first.
Figure \ref{fig:bddc-eig} reports the ratio $\kappa = \lambda_{\max}/\lambda_{\min}$ as a function of parameter $\omega$ for the two-level preconditioned operators with multiplicative Richardson/Jacobi relaxation, $G_{1,0}^f$ and $G_{2,0}^f$ as defined in Equation (\ref{Multi-fine-level}).
The zero eigenvalue resulting from periodic boundary conditions, if
identified by this procedure, was ignored.  In Figure
\ref{fig:bddc-eig}, the  optimal results of $G_{j,0}^f$ with $p=4$ and
$16$ match with LFA predictions in Tables \ref{Cond-2-level-G10f} and
\ref{Cond-2-level-G20f}.  Moreover, we see excellent agreement with
the LFA results presented in Figure \ref{two-grid-sensitivity} for the case of $p=4,8,16$.
\begin{figure}
\centering
\includegraphics[width=.8\textwidth]{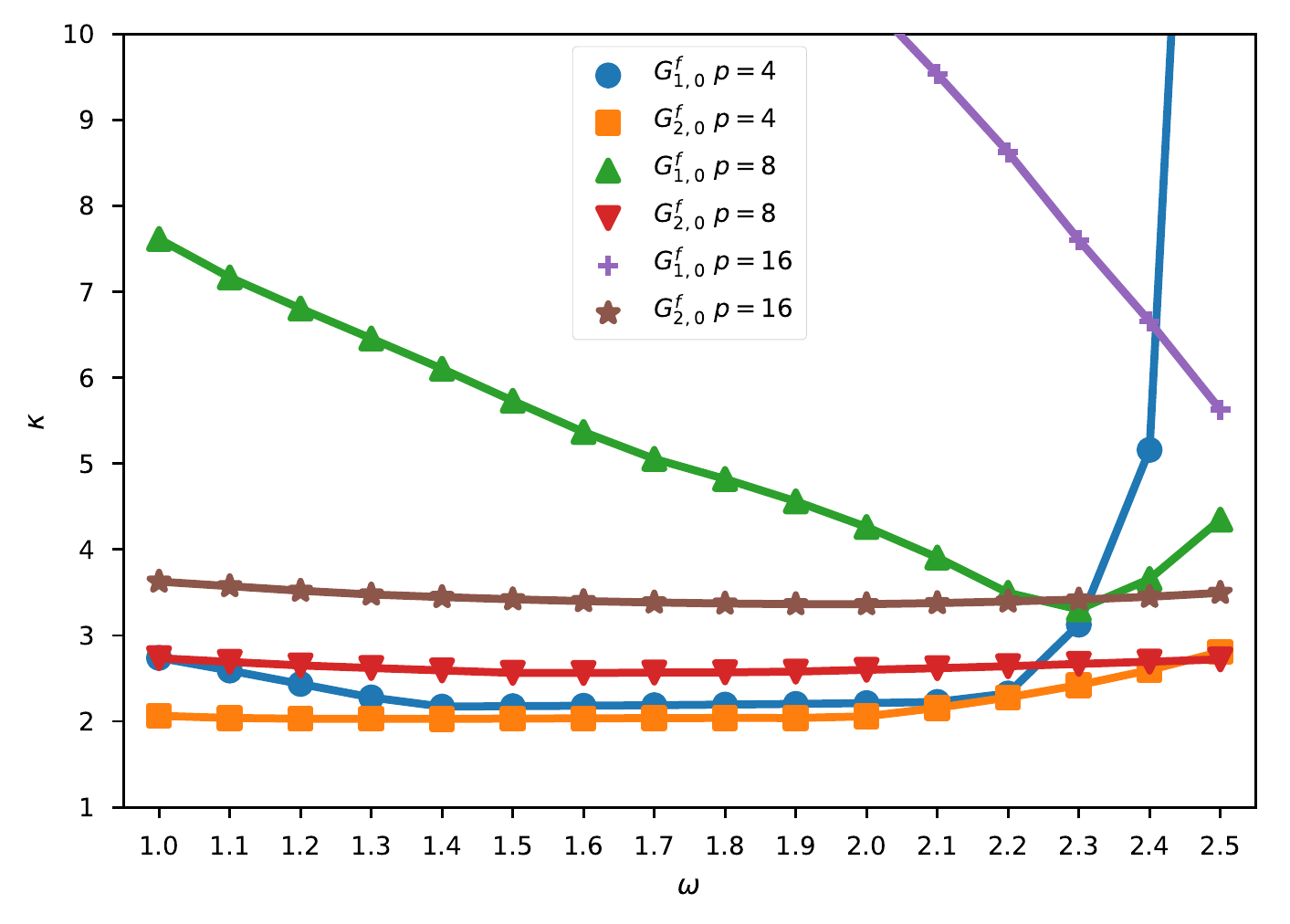}
\caption{Ratio of estimated eigenvalues $\kappa = \lambda_{\max}/\lambda_{\min}$ as a function of relaxation parameter $\omega$ for multiplicative two-level PCBDDC preconditioned operators using $16\times 16$ subdomains, each of size $p\times p$ with $p\in \{4, 8, 16\}$.  The results here are comparable to the LFA results of Figure \ref{two-grid-sensitivity}.} \label{fig:bddc-eig}
\end{figure}
}

\subsection{Eigenvalue Distribution of Two-level Variants}
\underline{}In this section, we take $n=32$, yielding $2n$ points in each dimension and $(2n)^2 = 4096$ values of $\boldsymbol{\theta}$, although similar results are seen for smaller values of $n$.  We also consider only $p=8$, although similar results are seen for other values of $p$.  For $\widetilde{G}_{1,0}^{f}$ and $\widetilde{G}_{2,0}^{f}$, we use the optimal values of $\omega$, shown in the tables above. The histograms in Figure \ref{Density-plot} show the density of eigenvalues for the two-level preconditioned operators. The y-axis is the ratio of the number of eigenvalues contained in a ``bin'' to its width, where the width of each bin  is 0.1.  For these values of $n$ and $p$, our LFA computes a total of $262144$ eigenvalues, giving $64$ eigenvalues for each of $4096$ sampling points. For all cases, the eigenvalues around 1 (represented in two bins in the histogram, covering the interval from 0.9 to 1.1) appear with dominating multiplicity, accounting for about 200,000 of the computed eigenvalues.

Note that there is a gap in the spectrum of $\widetilde{G}_{1,0}$ that increases in size with $p$ (not shown here).  A notable difference between $\widetilde{G}_{1,0}$ and $\widetilde{G}_{2,0}$ is that, while there is still a small gap in the spectrum of $\widetilde{G}_{2,0}$, it is not very prominent.  Note also that the spectra are real-valued, with only roundoff-level errors in the imaginary component.  Comparing the eigenvalues for $\widetilde{G}_{1,0}^{f}$ and $\widetilde{G}_{2,0}^{f}$ with those for $\widetilde{G}_{1,0}$ and $\widetilde{G}_{2,0}$, we see that the eigenvalues are much more tightly clustered for $\widetilde{G}_{1,0}^{f}$, but still exhibit a gap in the spectrum.  The eigenvalues of $\widetilde{G}_{2,0}^{f}$, in contrast, appear to lie in a continuous interval.  We note that little improvement is seen in the spectrum of $\tilde{G}_{2,0}^{f}$, in comparison with $\widetilde{G}_{2,0}$.  Also interesting to note is that, in contrast to all other cases, the smallest eigenvalue of $\widetilde{G}_{1,0}^{f}$ is less than 1.

\begin{remark}
As the LFA predicts both eigenvectors and eigenvalues, we can examine the frequency composition of the eigenvectors associated with these eigenvalues. The largest eigenvalue of $\widetilde{ G}_{1,0}$ is found to be dominated by oscillatory modes, but this is not true for $\widetilde{ G}_{2,0}$. This  motivates the proposed multiplicative method based on simple diagonal scaling, which is well known to effectively damp oscillatory errors in the classical multigrid setting.
\end{remark}

%

 \begin{figure}
\centering
\includegraphics[width=6.25cm,height=5.5cm]{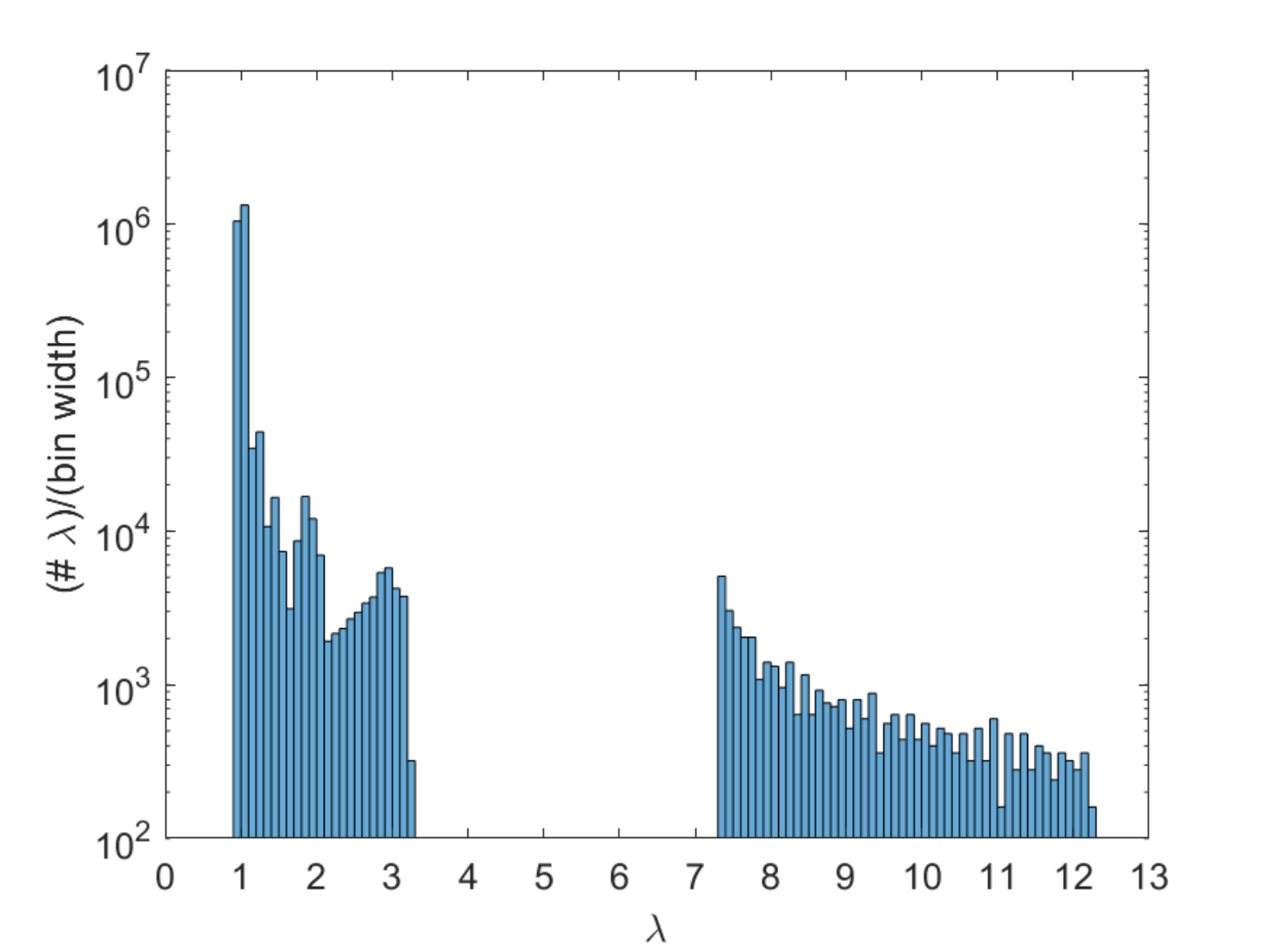}
\includegraphics[width=6.25cm,height=5.5cm]{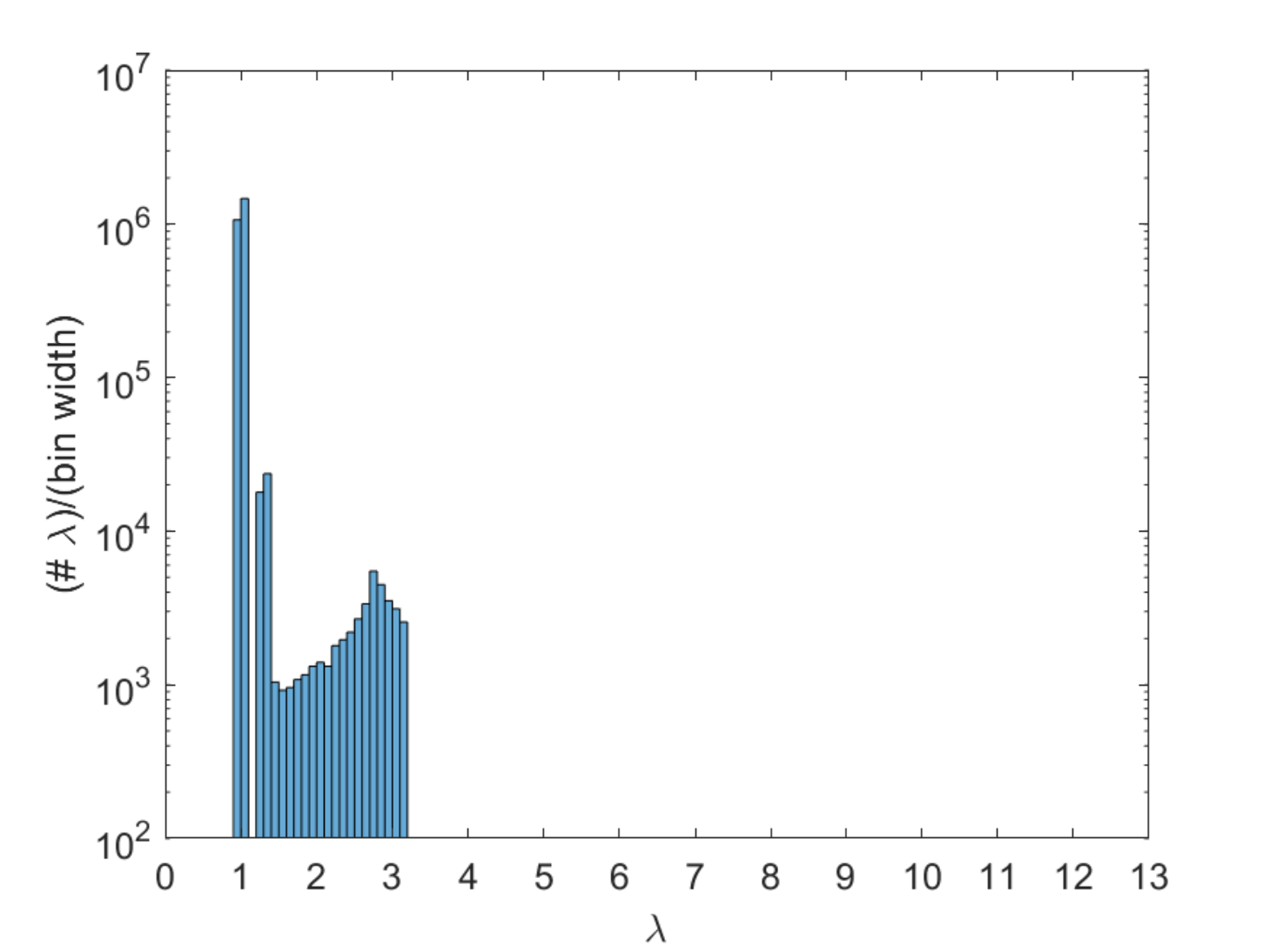}
\includegraphics[width=6.25cm,height=5.5cm]{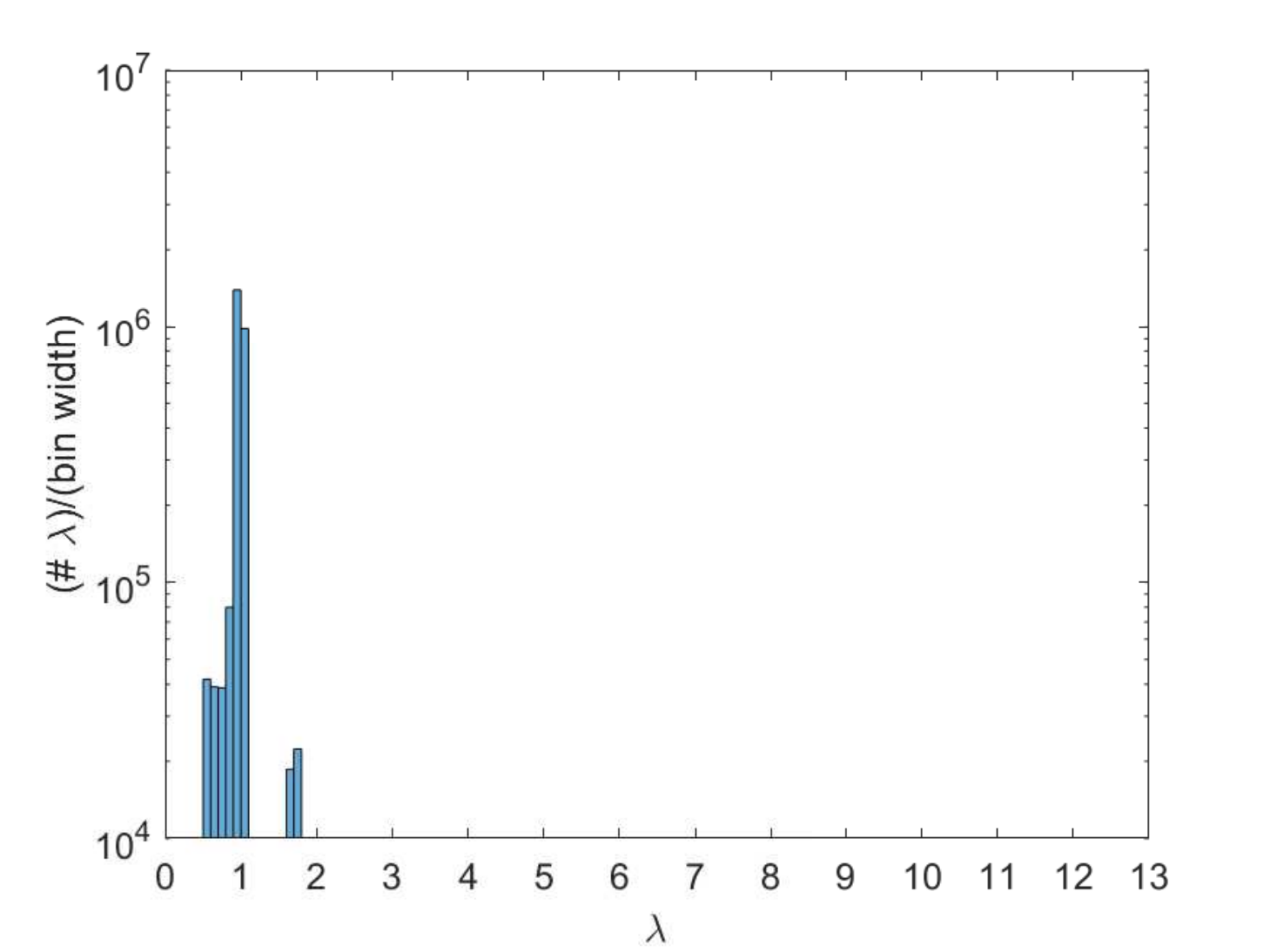}
\includegraphics[width=6.25cm,height=5.5cm]{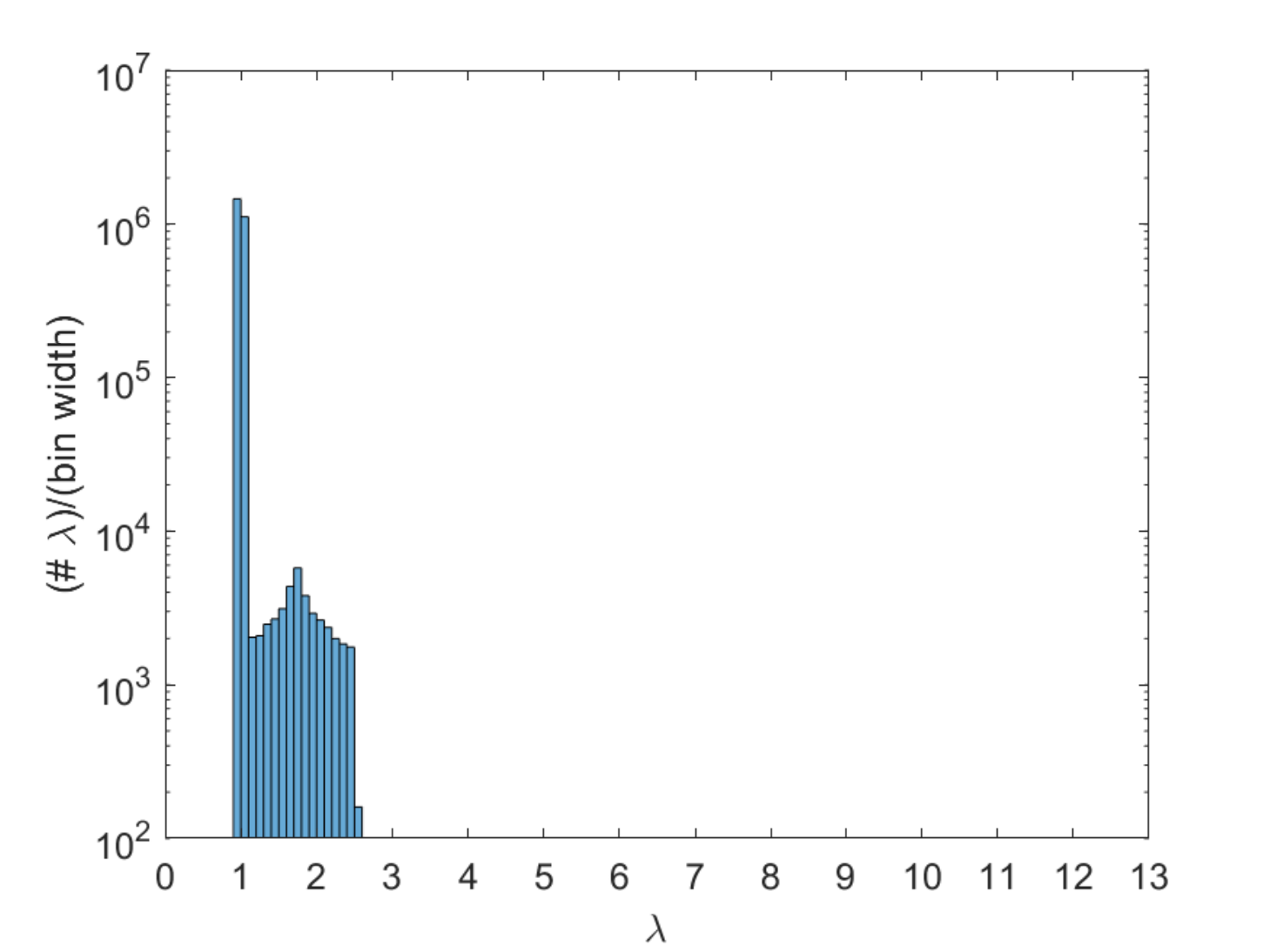}
\caption{Histograms showing density of eigenvalues for two-level preconditioned operators with $p=8$. Top left: $\widetilde{G}_{1,0}$, Top right: $\widetilde{G}_{2,0}$, Bottom left: $\widetilde{G}_{1,0}^{f}$, Bottom right: $\widetilde{G}_{2,0}^{f}$.} \label{Density-plot}
\end{figure}


\subsection{Condition Numbers of Three-level Variants}
For the three-level preconditioned operators, we need to find all the
eigenvalues of a $p^4\times p^4$ matrix for each sampled value of
$\boldsymbol{\theta}$. For the two-level variants, we saw that
sampling with $n=4$ is sufficient to give  useful accuracy of the LFA
predictions. Here, we also see similar behavior in Table
\ref{Cond-3-level}, which  shows the condition numbers (ratio of
extreme eigenvalues) of $\widetilde{G}_{i,j} (i,j=1,2)$ for varying $p$ and $n$. We see that, as expected from the theory, these condition numbers show degradation from the two-level case. It is not surprising that $\widetilde{G}_{2,2}$ has the smallest condition number of these variants, since $M_2$ is applied to both fine and coarse levels. Following Remark \ref{THM-Three-level}, we compute the constant, $\mathfrak{C}_{i,j}$, associated with the bound on $\kappa(\widetilde{G}_{i,j}) (i,j=1,2)$, see Table \ref{Cond-3-level} .  We see that the  constants needed to fit the theoretical bounds for $p=8$ are smaller than those for $p=4$. These suggest that those bounds may not be sharp. Similar behavior is seen  in \cite{zampini2017multilevel}, suggesting that the theoretical bounds may overestimate the true growth in the condition number.

{ As mentioned before, $G_{i,j}^c, G_{i,j}^{s,c}$ and
  $G_{i,j}^{f,c}$ may have complex eigenvalues. Here, we replace the
  condition number as a measure of the effectiveness of the
  preconditioner by the ratio of extreme eigenvalues, $\frac{{\rm
      max}|\lambda|}{{\rm min}|\lambda|}$}, noting that, except in
cases of large weights, the eigenvalues tend to remain clustered
around the positive real axis. Table \ref{Cond-Fine-Coarse-3-level}
presents these ratios for  variants $\widetilde{G}_{i,j}^{f}$ and $\widetilde{G}_{i,j}^{c}$,  based on the multiplicative combination with diagonal scaling on the fine level and coarse level, respectively, and some improvement is offered. For fixed $p$, the optimal $\omega$ is found to be robust to $n$ (not shown here). In general, we see better performance for $\widetilde{G}_{i,j}^{f}$ in comparison to $\widetilde{G}_{i,j}^{c}$, and $\widetilde{G}_{1,1}^{f}$ offers significant improvement over $\widetilde{G}_{1,1}$. For other values of $i,j$, however, only small improvements are seen.
\begin{table}
\centering
\caption{Ratios of extreme eigenvalues for three-level preconditioners with no multiplicative relaxation.}
\begin{tabular}{l||c|c|c|c}
\hline
\hline
$p$ &$\widetilde{G}_{1,1}$  &$\widetilde{G}_{1,2}$  &$\widetilde{G}_{2,1}$ &$\widetilde{G}_{2,2}$ \\
\hline
\hline
$4(n=2)$                  &9.18              &5.43       &7.27          &4.24   \\
$4(n=4)$                  &9.65               &5.68       &7.63         &4.47   \\
$4(n=8)$                  &9.79               &5.74       &7.73         &4.53    \\
$4(n=16)$                 &9.82              &5.76        &7.76         &4.54     \\
$4(n=32)$                 &9.83               &5.76       &7.77         &4.55     \\
\hline
$8(n=2)$                 &46.66               &15.46      &24.73        &7.55  \\
$8(n=4)$                 &50.00               &16.15      &26.53        &7.94  \\
$8(n=8)$                 &50.96               &16.33      &27.05        &8.04    \\
\hline
$\mathfrak{C}_{i,j}(p=4,n=32)$               & 0.11              & 0.11     & 0.14      &0.14    \\
\hline
$\mathfrak{C}_{i,j}(p=8,n=8)$               & 0.08              &  0.07     &0.12      &0.09    \\
\hline
\end{tabular}\label{Cond-3-level}
\end{table}

\begin{table}
\centering
\caption{Ratios of extreme eigenvalues for three-level preconditioners with fine-scale or coarse-scale multiplicative preconditioning. All results were computed with $n=4$, and the experimentally optimized weight, $\omega$, is shown in brackets.}
\begin{tabular}{l||c|c|c|c}
\hline
\hline
$p$  & $\widetilde{ G}_{1,1}^{f}$ &$\widetilde{G}_{1,2}^{f}$ &$\widetilde{G}_{2,1}^{f}$  &$\widetilde{G}_{2,2}^{f}$ \\
\hline
$4$                    &6.80(1.4)                  &4.28(1.4)              &6.14(1.6)       &4.04(1.1)     \\
\hline

$8$                    &28.75(1.7)                 &9.16(1.7)              &20.94(1.6)       &6.73(1.5)     \\
\hline
\multicolumn{5} {c} {\footnotesize{}} \\
\hline
\hline
$p$ & $\widetilde{ G}_{1,1}^{c}$  &$\widetilde{ G}_{1,2}^{c}$  &$\widetilde{G}_{2,1}^{c}$
 &$\widetilde{ G}_{2,2}^{c}$ \\
\hline

$4$                    &6.04(1.6)                  &5.47(1.1)              &4.67(1.6)        &4.30(1.0)     \\
\hline

$8$                    &31.91(2.0)                  &15.17(1.4)              &15.57(2.1)       &7.46(1.2)   \\
\hline
\hline
\end{tabular}\label{Cond-Fine-Coarse-3-level}
\end{table}


In order to see the sensitivity of performance to parameter choice, we
consider three-level preconditioners with weighted multiplicative
preconditioning on both fine and coarse scales,
$\widetilde{G}_{1,1}^{f,c}$ and $\widetilde{G}_{2,2}^{f,c}$, with
$p=4$ and $n=4$. At the left of Figure \ref{Gijfc-sensitivity}, we
present the LFA-predicted ratio of extreme eigenvalues for
$\widetilde{G}_{1,1}^{f,c}$ with variation in $\omega_1$ and
$\omega_2$. Here, we see strong sensitivity to ``small'' values of
$\omega_1$, for example $\omega_1 < 1.5$, and also to large values of
$\omega_1$ with small values of $\omega_2$.  We note general
improvement, though, in the optimal performance for large $\omega_1$
with suitably chosen $\omega_2$, albeit with diminishing returns as
$\omega_1$ continues to increase.  Fixing $\omega_1 = 4$, we find
$\omega_2 = 1.7$ offers best performance, with optimal eigenvalue
ratio of 2.66. At the right of Figure \ref{Gijfc-sensitivity},  we
consider $\widetilde{G}_{2,2}^{f,c}$ as a function of $\omega_1$ and
$\omega_2$.  Here, we see stronger sensitivity to large values of
$\omega_2$,  and to large values of  $\omega_1$ and small values of
$\omega_2$, but a large range of parameters that give generally
similar performance.  Fixing $\omega_1=4$, we find that
$\omega_2=1.2$ achieves the optimal eigenvalue ratio of 3.72. Similar
performance was seen for $\widetilde{G}_{1,2}^{f,c},\widetilde{
  G}_{2,1}^{f,c}$, and $\widetilde{G}_{i,j}^{s,c}$. Slight
improvements can be seen by allowing even larger values of $\omega_1$,
giving an LFA-predicted ratio of extreme eigenvalues for
$\widetilde{G}_{1,1}^{f,c}$ of 2.25 with $\omega_1 = 5.0$ and
$\omega_2 = 2.0$, but a much smaller band of values of $\omega_2$
leads to near-optimal performance as $\omega_1$ increases.  For
$\widetilde{G}_{2,2}^{f,c}$, this sensitivity does not arise, but the
improvements are even more marginal, achieving an LFA-predicted ratio
of extreme eigenvalues of 3.63 for $\omega_1 = 5.8$ and $\omega_2 = 1.3$.

Motivated by Figure \ref{Gijfc-sensitivity}, we fix $\omega_1=4$ with
$n=4$, and optimize the ratio of extreme eigenvalues for the three-level preconditioners with two multiplicative preconditioning steps per iteration, either both on the coarse level, $\widetilde{G}_{i,j}^{s,c}$, or one on each level, $\widetilde{G}_{i,j}^{f,c}$, with respect to $\omega_2$. From Table \ref{Cond-3-level-sym}, notable improvement is seen for all $i,j$ with $\widetilde{G}_{i,j}^{f,c}$, particularly for $\widetilde{G}_{1,1}^{f,c}$ and $\widetilde{G}_{2,1}^{f,c}$.  We also note that there is little variation in the optimal parameter for each preconditioner between the $p=4$ and $p=8$ cases. It is notable that we are able to achieve similar performance for the multiplicative preconditioner based on $M_1$ as seen for $M_2$, and that both show significant improvement from the classical three-level results shown in Table \ref{Cond-3-level}, when used in combination with multiplicative preconditioning on both fine and coarse levels.

 \begin{figure}
\centering
\includegraphics[width=6.25cm,height=5.5cm]{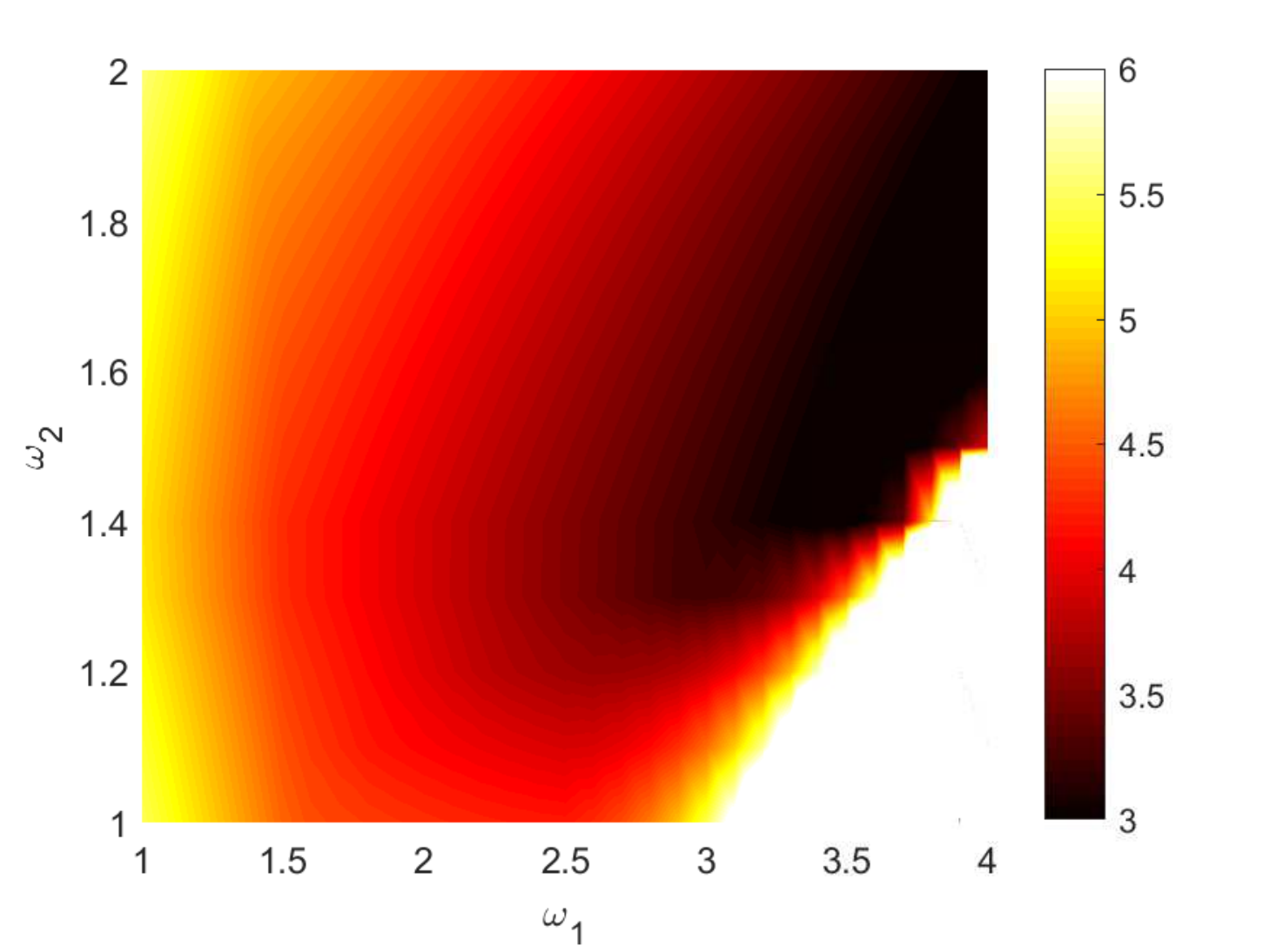}
\includegraphics[width=6.25cm,height=5.5cm]{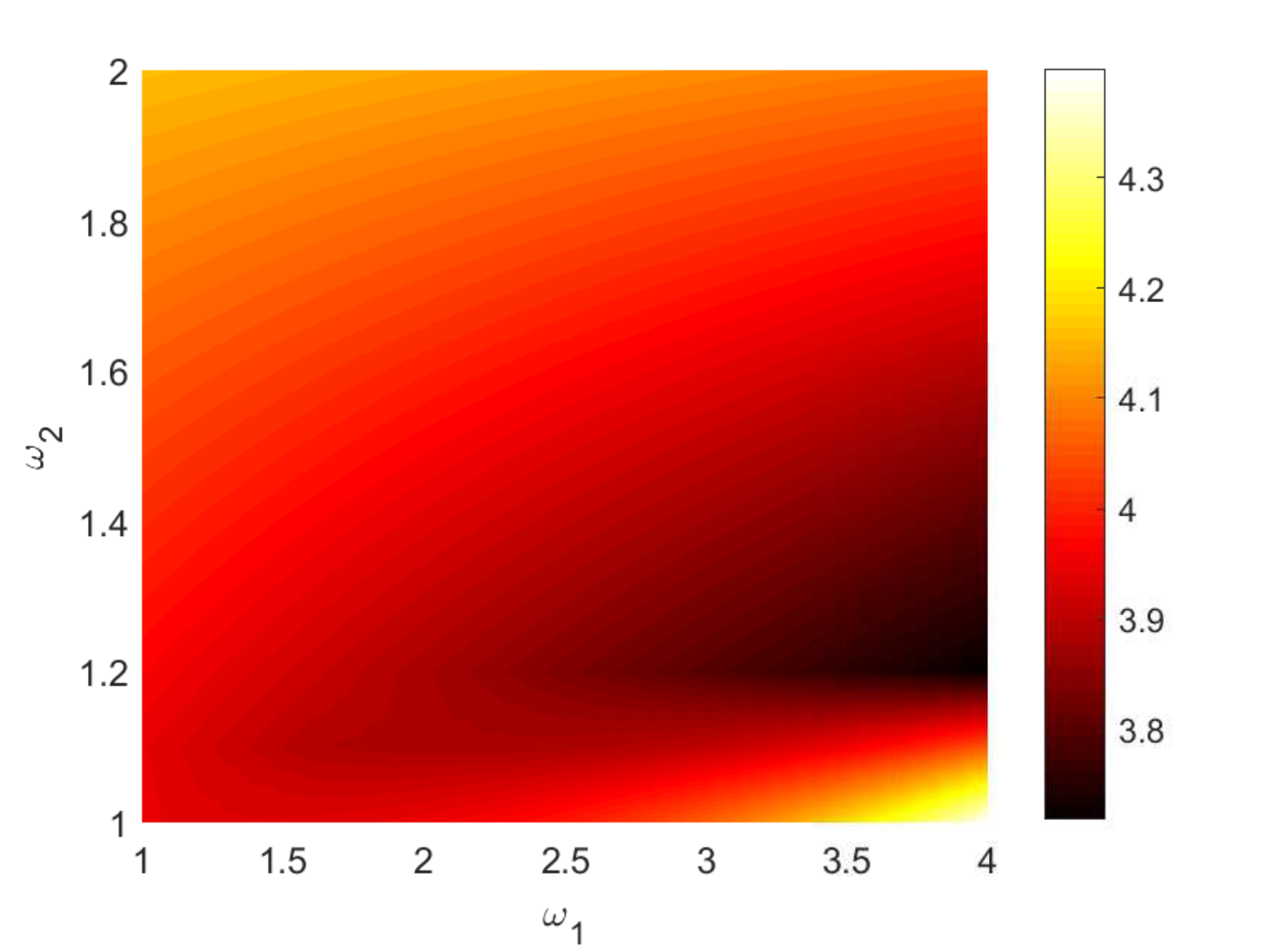}
\caption{Ratio of extreme eigenvalues for three-level preconditioners with multiplicative preconditioning on both the fine and coarse scales as a function of $\omega_1$ and $\omega_2$, with $p=4$ and $n=4$.  At left, ratio for $\widetilde{G}_{1,1}^{f,c}$; at right, ratio for $\widetilde{G}_{2,2}^{f,c}$.} \label{Gijfc-sensitivity}
\end{figure}

\begin{table}
\centering
\caption{Ratio of extreme eigenvalues for three-level preconditioners with symmetric weighting of multiplicative preconditioning on the coarse scale, $\widetilde{G}_{i,j}^{s,c}$, and weighting of multiplicative preconditioning on both fine and coarse scales, $\widetilde{G}_{i,j}^{f,c}$. All results were computed  with $n=4$, and the experimentally optimized weight, $\omega_2$, is shown in brackets.}
\begin{tabular}{l||c|c|c|c}
\hline
\hline
$p$ & $\widetilde{G}_{1,1}^{s,c}$  &$\widetilde{G}_{1,2}^{s,c}$  &$\widetilde{G}_{2,1}^{s,c}$
&$\widetilde{G}_{2,2}^{s,c}$ \\
\hline
$4$      &5.43(1.4)              &5.34(0.9)           &4.22(1.3)      &4.18(0.9)    \\
\hline
$8$      &17.45(1.2)             &14.13(1.0)         &8.31(1.1)      &6.88(0.9)   \\
\hline
\multicolumn{5} {c} {} \\
\hline
\hline
$p$   & $\widetilde{G}_{1,1}^{f,c}$      &$\widetilde{G}_{1,2}^{f,c}$      &$\widetilde{G}_{2,1}^{f,c}$ &$\widetilde{G}_{2,2}^{f,c}$ \\
\hline
$4$     &2.66(1.7)              &3.85(1.3)             &3.24(1.8)       &3.72(1.2)     \\
\hline
$8$     &5.16(1.8)              &7.59(1.7)          &4.88(1.8)    &5.70(1.5)    \\
\hline
\hline
\end{tabular}\label{Cond-3-level-sym}
\end{table}

\section{Conclusions} \label{sec:concl}
In this paper, we quantitatively estimate the condition numbers of variants of BDDC algorithms, using local Fourier analysis. A modified choice of basis is used to simplify the LFA, and we believe this choice will prove useful in analysing many domain decomposition algorithms in the style used here. Multiplicative preconditioners with these two domain decomposition methods are discussed briefly, and both lumped and Dirichlet variants can be improved in this way.  The coarse problem involved in these domain decomposition methods can be solved by similar methods. LFA analysis of three-level variants is also considered. Degradation in convergence is well known when moving from two-level to three-level variants of these algorithms. We  show that the LFA presented above, in combination with the use of multiplicative preconditioners on the coarse and fine levels provide ways to mitigate this performance loss. Future work includes extending these variants of the preconditioned operators, using LFA to optimize the resulting algorithms, and considering other types of problems with similar preconditioners.



\bibliographystyle{siamplain}
\bibliography{BDDC}

\begin{thebibliography}{10}

\bibitem{badia2018fempar}
{\sc S.~Badia, A.~F. Mart{\'\i}n, and J.~Principe}, {\em {FEMPAR}: An
  object-oriented parallel finite element framework}, Archives of Computational
  Methods in Engineering, 25 (2018), pp.~195--271.

\bibitem{petsc-user-ref}
{\sc S.~Balay, S.~Abhyankar, M.~F. Adams, J.~Brown, P.~Brune, K.~Buschelman,
  L.~Dalcin, A.~Dener, V.~Eijkhout, W.~D. Gropp, D.~Kaushik, M.~G. Knepley,
  D.~A. May, L.~C. McInnes, R.~T. Mills, T.~Munson, K.~Rupp, P.~Sanan, B.~F.
  Smith, S.~Zampini, H.~Zhang, and H.~Zhang}, {\em {PETS}c users manual:
  Revision 3.10}, Tech. Report ANL-95/11 - Rev 3.10, Argonne National
  Laboratory, 2018.

\bibitem{bolten2018fourier}
{\sc M.~Bolten and H.~Rittich}, {\em Fourier analysis of periodic stencils in
  multigrid methods}, SIAM Journal on Scientific Computing, 40 (2018),
  pp.~A1642--A1668.

\bibitem{boonen2008local}
{\sc T.~Boonen, J.~Van~Lent, and S.~Vandewalle}, {\em Local {F}ourier analysis
  of multigrid for the curl-curl equation}, SIAM Journal on Scientific
  Computing, 30 (2008), pp.~1730--1755.

\bibitem{MR1090464}
{\sc J.~H. Bramble, J.~E. Pasciak, J.~P. Wang, and J.~Xu}, {\em Convergence
  estimates for product iterative methods with applications to domain
  decomposition}, Mathematics of Computation, 57 (1991), pp.~1--21.

\bibitem{MR0431719}
{\sc A.~Brandt}, {\em Multi-level adaptive solutions to boundary-value
  problems}, Mathematics of Computation, 31 (1977), pp.~333--390.

\bibitem{MR3666858}
{\sc S.~C. Brenner, E.-H. Park, and L.-Y. Sung}, {\em A {BDDC} preconditioner
  for a symmetric interior penalty {G}alerkin method}, Electronic Transactions
  on Numerical Analysis, 46 (2017), pp.~190--214.

\bibitem{brenner2007bddc}
{\sc S.~C. Brenner and L.-Y. Sung}, {\em {BDDC} and {FETI-DP} without matrices
  or vectors}, Computer Methods in Applied Mechanics and Engineering, 196
  (2007), pp.~1429--1435.

\bibitem{briggs2000multigrid}
{\sc W.~L. Briggs, V.~E. Henson, and S.~F. McCormick}, {\em A multigrid
  tutorial}, SIAM, 2000.

\bibitem{dohrmann2003preconditioner}
{\sc C.~R. Dohrmann}, {\em A preconditioner for substructuring based on
  constrained energy minimization}, SIAM Journal on Scientific Computing, 25
  (2003), pp.~246--258.

\bibitem{MR2334091}
{\sc C.~R. Dohrmann}, {\em Preconditioning of saddle point systems by
  substructuring and a penalty approach}, in Domain decomposition methods in
  science and engineering {XVI}, vol.~55 of Lecture Notes in Computational
  Science and Engineering, Springer, Berlin, 2007, pp.~53--64.

\bibitem{dohrmann2016bddc}
{\sc C.~R. Dohrmann and O.~B. Widlund}, {\em A {BDDC} algorithm with deluxe
  scaling for three-dimensional {$H(\mathrm{curl})$} problems}, Communications
  on Pure and Applied Mathematics, 69 (2016), pp.~745--770.

\bibitem{MR3450068}
{\sc V.~Dolean, P.~Jolivet, and F.~Nataf}, {\em An introduction to domain
  decomposition methods: {A}lgorithms, theory, and parallel implementation},
  Society for Industrial and Applied Mathematics (SIAM), Philadelphia, PA,
  2015.

\bibitem{MR2372024}
{\sc M.~Dryja, J.~Galvis, and M.~Sarkis}, {\em B{DDC} methods for discontinuous
  {G}alerkin discretization of elliptic problems}, Journal of Complexity, 23
  (2007), pp.~715--739.

\bibitem{MR1064335}
{\sc M.~Dryja and O.~B. Widlund}, {\em Towards a unified theory of domain
  decomposition algorithms for elliptic problems}, in Third {I}nternational
  {S}ymposium on {D}omain {D}ecomposition {M}ethods for {P}artial
  {D}ifferential {E}quations ({H}ouston, {TX}, 1989), SIAM, Philadelphia, PA,
  1990, pp.~3--21.

\bibitem{dryja2002feti}
{\sc M.~Dryja and O.~B. Widlund}, {\em A {FETI-DP} method for a mortar
  discretization of elliptic problems}, Lecture Notes in Computational Science
  and Engineering, 23 (2002), pp.~41--52.

\bibitem{farhat2001feti}
{\sc C.~Farhat, M.~Lesoinne, P.~LeTallec, K.~Pierson, and D.~Rixen}, {\em
  {FETI-DP}: a dual--primal unified {FETI} method {P}art {I}: {A} faster
  alternative to the two-level {FETI} method}, International Journal for
  Numerical Methods in Engineering, 50 (2001), pp.~1523--1544.

\bibitem{friedhoff2015generalized}
{\sc S.~Friedhoff and S.~MacLachlan}, {\em A generalized predictive analysis
  tool for multigrid methods}, Numerical Linear Algebra with Applications, 22
  (2015), pp.~618--647.

\bibitem{friedhoff2013local}
{\sc S.~Friedhoff, S.~MacLachlan, and C.~Borgers}, {\em Local {F}ourier
  analysis of space-time relaxation and multigrid schemes}, SIAM Journal on
  Scientific Computing, 35 (2013), pp.~S250--S276.

\bibitem{MR1924414}
{\sc M.~J. Gander, F.~Magoul\`es, and F.~Nataf}, {\em Optimized {S}chwarz
  methods without overlap for the {H}elmholtz equation}, SIAM Journal on
  Scientific Computing, 24 (2002), pp.~38--60.

\bibitem{kumar2018cell}
{\sc P.~Kumar, C.~Rodrigo, F.~J. Gaspar, and C.~W. Oosterlee}, {\em On local
  {F}ourier analysis of multigrid methods for {PDE}s with jumping and random
  coefficients}, SIAM Journal on Scientific Computing, 41 (2019),
  pp.~A1385--A1413.

\bibitem{kuo1989two}
{\sc C.-C.~J. Kuo and B.~C. Levy}, {\em Two-color {F}ourier analysis of the
  multigrid method with red-black {G}auss-{S}eidel smoothing}, Applied
  Mathematics and Computation, 29 (1989), pp.~69--87.

\bibitem{li2005dual}
{\sc J.~Li}, {\em A dual-primal {FETI} method for incompressible {S}tokes
  equations}, Numerische Mathematik, 102 (2005), pp.~257--275.

\bibitem{li2006bddc}
{\sc J.~Li and O.~Widlund}, {\em {BDDC} algorithms for incompressible {S}tokes
  equations}, SIAM Journal on Numerical Analysis, 44 (2006), pp.~2432--2455.

\bibitem{li2006feti}
{\sc J.~Li and O.~Widlund}, {\em {FETI-DP}, {BDDC}, and block {C}holesky
  methods}, International Journal for Numerical Methods in Engineering, 66
  (2006), pp.~250--271.

\bibitem{MR2334130}
{\sc J.~Li and O.~Widlund}, {\em A {BDDC} preconditioner for saddle point
  problems}, in Domain decomposition methods in science and engineering {XVI},
  vol.~55 of Lecture Notes in Computational Science and Engineering, Springer,
  Berlin, 2007, pp.~413--420.

\bibitem{li2007use}
{\sc J.~Li and O.~Widlund}, {\em On the use of inexact subdomain solvers for
  {BDDC} algorithms}, Computer Methods in Applied Mechanics and Engineering,
  196 (2007), pp.~1415--1428.

\bibitem{maclachlan2011local}
{\sc S.~P. MacLachlan and C.~W. Oosterlee}, {\em Local {F}ourier analysis for
  multigrid with overlapping smoothers applied to systems of {PDE}s}, Numerical
  Linear Algebra with Applications, 18 (2011), pp.~751--774.

\bibitem{mandel2003convergence}
{\sc J.~Mandel and C.~R. Dohrmann}, {\em Convergence of a balancing domain
  decomposition by constraints and energy minimization}, Numerical Linear
  Algebra with Applications, 10 (2003), pp.~639--659.

\bibitem{mandel2005algebraic}
{\sc J.~Mandel, C.~R. Dohrmann, and R.~Tezaur}, {\em An algebraic theory for
  primal and dual substructuring methods by constraints}, Applied Numerical
  Mathematics, 54 (2005), pp.~167--193.

\bibitem{MR2457352}
{\sc J.~Mandel, B.~r. Soused\'ik, and C.~R. Dohrmann}, {\em Multispace and
  multilevel {BDDC}}, Computing. Archives for Scientific Computing, 83 (2008),
  pp.~55--85.

\bibitem{pavarino2010bddc}
{\sc L.~F. Pavarino, O.~B. Widlund, and S.~Zampini}, {\em {BDDC}
  preconditioners for spectral element discretizations of almost incompressible
  elasticity in three dimensions}, SIAM Journal on Scientific Computing, 32
  (2010), pp.~3604--3626.

\bibitem{rodrigo2012multicolor}
{\sc C.~Rodrigo, F.~J. Gaspar, and F.~J. Lisbona}, {\em Multicolor {F}ourier
  analysis of the multigrid method for quadratic {FEM} discretizations},
  Applied Mathematics and Computation, 218 (2012), pp.~11182--11195.

\bibitem{sousedik2013adaptive}
{\sc B.~Soused{\'\i}k, J.~{\v{S}}{\'\i}stek, and J.~Mandel}, {\em
  Adaptive-multilevel {BDDC} and its parallel implementation}, Computing, 95
  (2013), pp.~1087--1119.

\bibitem{stuben1982multigrid}
{\sc K.~St{\"u}ben and U.~Trottenberg}, {\em Multigrid methods: {F}undamental
  algorithms, model problem analysis and applications}, Multigrid Methods,
  (1982), pp.~1--176.

\bibitem{MR2104179}
{\sc A.~Toselli and O.~Widlund}, {\em Domain decomposition methods: algorithms
  and theory}, vol.~34 of Springer Series in Computational Mathematics,
  Springer-Verlag, Berlin, 2005.

\bibitem{MR1807961}
{\sc U.~Trottenberg, C.~W. Oosterlee, and A.~Sch{\"u}ller}, {\em Multigrid},
  Academic Press, Inc., San Diego, CA, 2001.
\newblock With contributions by A. Brandt, P. Oswald and K. St{\"u}ben.

\bibitem{tu2006bddc}
{\sc X.~Tu}, {\em {D}omain decomposition algorithms: methods with three levels
  and for flow in porous media}, PhD thesis, New York University, Graduate
  School of Arts and Science, 2006.

\bibitem{MR2341811}
{\sc X.~Tu}, {\em Three-level {BDDC} in three dimensions}, SIAM Journal on
  Scientific Computing, 29 (2007), pp.~1759--1780.

\bibitem{MR2282536}
{\sc X.~Tu}, {\em Three-level {BDDC} in two dimensions}, International Journal
  for Numerical Methods in Engineering, 69 (2007), pp.~33--59.

\bibitem{MR1156079}
{\sc P.~Wesseling}, {\em An introduction to multigrid methods}, Pure and
  Applied Mathematics (New York), John Wiley \& Sons, Ltd., Chichester, 1992.

\bibitem{wienands2004practical}
{\sc R.~Wienands and W.~Joppich}, {\em Practical {F}ourier analysis for
  multigrid methods}, CRC press, 2004.

\bibitem{zampini2016pcbddc}
{\sc S.~Zampini}, {\em {PCBDDC}: a class of robust dual-primal methods in
  {PETS}c}, SIAM Journal on Scientific Computing, 38 (2016), pp.~S282--S306.

\bibitem{zampini2017multilevel}
{\sc S.~Zampini and X.~Tu}, {\em Multilevel balancing domain decomposition by
  constraints deluxe algorithms with adaptive coarse spaces for flow in porous
  media}, SIAM Journal on Scientific Computing, 39 (2017), pp.~A1389--A1415.

\end{thebibliography}
\end{document}